\documentclass[11pt]{article}
\usepackage{amsmath,amsfonts}
\usepackage{verbatim}
\usepackage{latexsym}
\usepackage{graphicx}
\usepackage{float}
\usepackage{color}
\usepackage{cite}
\allowdisplaybreaks
\makeatletter
\@addtoreset{equation}{section}

\begin{document}

\newcommand{\E}{\mathbb{E}}
\newcommand{\PP}{\mathbb{P}}
\newcommand{\RR}{\mathbb{R}}

\newtheorem{thm}{Theorem}[section]
\newtheorem{cor}[thm]{Corollary}
\newtheorem{lem}[thm]{Lemma}
\newtheorem{prp}[thm]{Proposition}
\newtheorem{ass}[thm]{Assumption}
\newtheorem{rem}{Remark}
\newtheorem{prf}[thm]{proof}
\newtheorem{exa}[thm]{Example}
\newtheorem{defn}{Definition}[section]

\newcommand\tq{{\scriptstyle{3\over 4 }\scriptstyle}}
\newcommand\qua{{\scriptstyle{1\over 4 }\scriptstyle}}
\newcommand\hf{{\textstyle{1\over 2 }\displaystyle}}
\newcommand\hhf{{\scriptstyle{1\over 2 }\scriptstyle}}

\newcommand{\proof}{\noindent {\it Proof}. }
\newcommand{\eproof}{\hfill $\Box$} 

\def\tl{\tilde}
\def\trace{\hbox{\rm trace}}
\def\diag{\hbox{\rm diag}}
\def\for{\quad\hbox{for }}
\def\refer{\hangindent=0.3in\hangafter=1}

\newcommand\wD{\widehat{\D}}
\title{
\bf The truncated EM scheme for multiple-delay SDEs with irregular coefficients and application to stochastic volatility  model
}

\author{
{\bf Zhuoqi Liu${}^a$,  Zhaohang Wang${}^b$, Siying Sun${}^a$, Shuaibin Gao${}^a$}
\\
${}^a$ Department of Mathematics, \\
Shanghai Normal University, \\
Shanghai, 200234, China. \\
${}^b$ School of Mathematics and Statistics, \\
South-Central Minzu University, \\
Wuhan, 430074, China.\\
 }

\date{}

\maketitle

\begin{abstract}
This paper focuses on the numerical scheme for multiple-delay stochastic differential equations with partially H\"older continuous drifts and locally H\"older continuous diffusion coefficients. To handle with the superlinear terms in coefficients, the truncated Euler-Maruyama scheme is employed. Under the given conditions, the convergence rates at time $T$ in both $\mathcal{L}^{1}$ and $\mathcal{L}^{2}$ senses are shown by virtue of the Yamada-Watanabe approximation technique. Moreover,  the convergence rates over a finite time interval $[0,T]$ are also obtained. Additionally, it should be noted that the convergence rates will not be affected by the number of delay variables. Finally, we perform the numerical experiments on the stochastic volatility  model to verify the reliability of the theoretical results.
 \end{abstract}
\section{Introduction}\label{sec1}

Due to their widespread applications, stochastic differential equations (SDEs) are attractive to researchers in recent decades \cite{1,2,20}. 
The numerical solutions approximated by Euler-Maruyama (EM) schemes have been widely investigated, since plenty of SDEs cannot been solved explicitly \cite{50}.
However, \cite{13} demonstrated that under superlinear growing condition, the $p$th moment of classical EM approximation will diverge to infinity in a finite time for $p\geq1$.
Thence, some implicit schemes have been developed to approximate SDEs with superlinear coefficients \cite{52,25,51}. Owing to the advantages of explicit schemes, several kinds of modified EM schemes also have been proposed to solve such SDEs, such as truncated EM scheme (TEMS) \cite{15,22,23}, tamed EM scheme \cite{14,26}, projected EM scheme \cite{57}, multilevel EM scheme \cite{56}, stopped EM scheme \cite{55} and so on.

Moreover, the Cox-Ingersoll-Ross (CIR) models have been rapidly developed in recent decades, since they are widely used in biology and finance \cite{xin20,xin22,xin24,xin21,xin23}. 
As is well known,  the characteristic of CIR model is that  the diffusion coefficients are H\"older continuous.
And the researches on numerical schemes for SDEs with H\"older continuous diffusion coefficients (SDEwHCDC) include the following papers:
the EM approximations for SDEwHCDC were established in \cite{xin18,xin1,58,xin19};
the strong convergence rate of the tamed EM scheme for SDEwHCDC and superlinear drift coefficients was studied in \cite{28};
the strong convergence rates of the TEMS for SDEwHCDC and superlinear drift coefficients were analyzed in \cite{31,29}.
When the diffusion coefficients are locally H\"older continuous, the tamed EM schemes \cite{7,16}  and tamed-adaptive EM scheme \cite{xin32} were discussed for superlinear SDEs. 
As for McKean-Vlasov SDEwHCDC, the numerical schemes were shown in \cite{xin12,xin31,xin11}.

Before the description of stochastic volatility  model in Section 6, let us focus on the following financial model called risk-neutral process in \cite{60}:
\begin{equation}\label{VIX1}
	dV(t)=\big(c_{4}V(t)+c_{5}V^{2}(t)-\lambda^{*}kV^\frac{3}{2}(t)\big)dt+kV^\frac{3}{2}(t)dB(t),
\end{equation}
where $B(t)$ is a  Brownian motion, $V(t)$ is the volatility and  $\lambda^{*}$ is the risk related to the Volatility Index (VIX), which can minimize the error between market and model prices. Here, assume that $c_4, c_5, k$ are constants.
The characteristic of (\ref{VIX1}) is that the diffusion coefficient and a part of the drift coefficient are both locally H\"older continuous.
However, we have not found result about numerical scheme for such SDEs yet.
To facilitate research, we rewrite (\ref{VIX1}) as:
\begin{equation}\label{vv2}
	\begin{split}
		dz(t)=&\big(a_1z(t)|z(t)|+a_2z^{3}(t)+a_3z(t)+a_4z(t-\tau_1)
		+a_5z(t-\tau_2)+a_6z(t)|z(t)|^\frac{1}{2}\big)dt+a_7|z(t)|^\frac{3}{2}dB(t),
	\end{split}
\end{equation}
with the initial data $\xi$, where the influence of time delay is taken into account.
Here, $\tau_1,\tau_2$ are constant delays and $a_{j}$ are constants for $j=1,2,\cdots,7$.
Obviously, the main characteristic of (\ref{VIX1}) has been preserved.
It is crucial and significative to analyze the numerical scheme for such SDEs.
On the other hand, the systems with multiple delays appear in \cite{xin13,xin14,xin15,xin16} and we cannot ignore the effect of multiple delays on the systems.  
As for the results about numerical schemes for SDEs with time delay, we refer the readers to \cite{3,xin17,5,6,9,18,53,x00xinx2,x00xinx3} for more details.

Based on the above discussions, the aim of this paper is to investigate the numerical scheme for the general form of multiple-delay SDEs (MDSDEs) with irregular coefficients, in which (\ref{vv2}) is contained. The contributions are stated as follows.
\begin{itemize}
	\item[$\bullet$] The TEMS for superlinear MDSDEs with partially H\"older continuous drifts and locally H\"older continuous diffusion coefficients is established.

	\item[$\bullet$] The convergence rates at time $T$ in both $\mathcal{L}^{1}$ and $ \mathcal{L}^{2}$ sense are given.

	\item[$\bullet$]  The convergence rates over a finite time interval $[0,T]$ in both $\mathcal{L}^{1}$ and $ \mathcal{L}^{2}$ sense are also revealed, which plays a vital role in approximating the European barrier option value \cite{xin88}.
	
	\item[$\bullet$] The numerical simulation for the stochastic volatility  model  is presented. Moreover, the conclusion that the convergence rates will not be affected by the number of delay variables is also verified.
\end{itemize}


%

This paper is organized as follows.  In section 2, we introduce some necessary notations, assumptions and the Yamada-Watanabe approximation technique. In section 3, the TEMS for MDSDE is established, and the boundedness of numerical solution is shown. Section 4 gives the convergence rates of the TEMS at time $T$ in $\mathcal{L}^{1} $and $\mathcal{L}^{2}$ sense. Section 5 demonstrates the convergence rates over a finite time interval $[0,T]$. In Section 6, we use the numerical experiment to validate the  feasibility of the theoretical results.


\section{Preliminaries}
Let $|\cdot |$ denote the Euclidean norm of vectors in $\mathbb{R}$. 
For real numbers $a$, $b$, let $a\wedge b=\min\{a,b\}$ and $a\vee b=\max\{a,b\}$, and denote $\lfloor a\rfloor$ the integer part of $a$. 
For a set $H$, define the indicator function $\mathbb{I}_H$: 
$\mathbb{I}_H(x) = 1$ if $x \in H$ and $\mathbb{I}_H(x) = 0$ if $x \notin H$.
Let $\big ( \Omega, \mathcal{F}, \mathbb{P} \big )$ stand for a complete probability space with a filtration  $\{\mathcal{F}_t\}_{t\in[0 , T]}$ satisfying the usual conditions and $\mathbb{E}$ be the probability expectation with respect to (w.r.t.) $\mathbb{P}$. 
Let $\tau$ be the time delay constant. Let $\mathcal{C}=\mathcal{C} ([-\tau,0];\mathbb{R})$ be the family of all continuous functions from  $[-\tau,0]$  to   $\mathbb{R}$, endowed with the  norm $||\xi||:=\sup_{-\tau\leq \theta\leq 0}|\xi(\theta)|$. 
Let $\mathcal{L}^q=\mathcal{L}^q\big ( \Omega, \mathcal{F}, \mathbb{P} \big )$ 
be the space of random variables $X$ with $\mathbb{E}|X|^q<\infty$ for $q \geq 1$.

Consider the superlinear MDSDEs with irregular coefficients of the form:
\begin{equation}\label{msde}
	dz(t)=\alpha\left( z(t), z(t-\tau_2),\cdots,z(t-\tau_r)\right)dt
	+\beta\left(z(t)\right) dB(t),
\end{equation}
with the initial data
$\xi=\{ \xi( \theta ) : - \tau \leq \theta \leq 0 \}\in \mathcal{C},$
where
$\alpha :(\mathbb{R})^r \rightarrow \mathbb{R}$ and  $\beta :\mathbb{R} \rightarrow \mathbb{R}$ are measurable. Here, $B(t)$ is an one-dimensional Brownian motion and $\tau_2, \tau_3, \cdots, \tau_r$ are all delay constants. 

The segment process $z_{t}$ is defined as $z_{t}=\{z(t+\theta):-\tau\leq\theta\leq0\}$.
Next, the projection operator is presented to make the prove process more simple. Define projection operator
$\varPhi_{\theta}(\zeta)=\zeta(\theta)$
for $\theta\in[-\tau,0] $ and $\zeta \in \mathcal{C}$.

For $v\in\mathbb{S}_{r}:=\{1,2,...,r\}$, we arrange
the sequence $\{\tau_{v}\}$ to make it be a non-negative and increasing sequence and let $\tau_{1}=0, \tau_r=\tau$. 
In order not to cause ambiguity, we make a sequence $\{\bar{s}_{v}\}$ satisfying $\tau_{v}=-\bar{s}_{v}$ for  $v\in\mathbb{S}_{r}$. 
Hence, $\bar{s}_{1}=-\tau_{1}=0, \bar{s}_{r}=-\tau_r=-\tau$. 
Then for $\zeta\in\mathcal{C}$, define $$\varPhi(\zeta)=\Big(\varPhi_{\bar{s}_{1}}(\zeta),\varPhi_{\bar{s}_{2}}(\zeta),\cdots,\varPhi_{\bar{s}_{r}}(\zeta)\Big).$$
Referring to the notations above leads to
\begin{equation*}
	\begin{split}
		\varPhi(z_t)=&\Big(\varPhi_{\bar{s}_{1}}(z_t),\varPhi_{\bar{s}_{2}}(z_t),\cdots,\varPhi_{\bar{s}_{r}}(z_t)\Big)
		=\Big(z_t(\bar{s}_{1}),z_t(\bar{s}_{2}),\cdots,z_t(\bar{s}_{r})\Big)\\
		=&\Big(z(t+\bar{s}_{1}),z(t+\bar{s}_{2}),\cdots,z(t+\bar{s}_{r})\Big)
		=\Big(z(t),z(t-\tau_{2}),\cdots,z(t-\tau)\Big).
	\end{split}
\end{equation*}
Then the target equation (\ref{msde}) becomes to
\begin{equation}\label{msdech}
	dz(t)=\alpha\left( \varPhi(z_t)\right)dt
	+\beta\left(z(t)\right) dB(t),
\end{equation}
with the initial data $\xi$, where
$\alpha :(\mathbb{R})^r \rightarrow \mathbb{R}$ and  $\beta :\mathbb{R} \rightarrow \mathbb{R}$ are measurable.
We need to impose some conditions on the coefficients. Denote $m^{(r)}=\left(m_1,m_2,\cdots,m_r\right)$, $n^{(r)}=\left(n_1,n_2,\cdots,n_r\right)$ for $m_i$, $n_i$ $\in\mathbb{R}$, $i\in\mathbb{S}_r$.	
Suppose that the drift coefficient can be decomposed as $\alpha(m^{(r)})=\alpha_{1}(m_{1})+\alpha_{2}(m_1)+\alpha_{3}(m^{(r)})$. 
\begin{ass}\label{a1}
	There exist constants $L_1>0$, $l_{1}\geq1$ and $\eta\in(0,1)$ such that
	$$\left|\alpha_{1}(m_{1})-\alpha_{1}(n_{1})\right|	\leq L_{1}\left(1+|m_1|^{l_{1}}+|n_{1}|^{l_{1}}\right)|m_{1}-n_{1}|^\eta,$$
	$$(m_1-n_{1}) \left( \alpha_{2}(m_{1})-\alpha_{2}(n_{1})\right) 
	\leq L_{1}|m_1-n_{1}|^2,$$
	$$\left|\alpha_{3}(m^{(r)})-\alpha_{3}(n^{(r)})\right|	\leq L_{1}\sum_{i=1}^{r}|m_i-n_{i}|,$$	
	for any $m_i, n_i\in \mathbb{R}$,  $i\in\mathbb{S}_r$. Here, $\alpha_{1}(\cdot)$ is a continuous and non-increasing function.
\end{ass}

\begin{ass} \label{a2}
	There exist constants $\bar{L}_2>0$ and $\check{p}\geq2$ such that
	$$m_1\big( \alpha_{1}(m_{1})+\alpha_{2}(m_1)\big)+\frac{\check{p}-1}{2}\left| \beta(m_1)\right|^2 
	\leq \bar{L}_{2}(1+|m_1|^2),$$
	for any $m_i\in \mathbb{R}$, $i\in\mathbb{S}_r$.
\end{ass}

By Assumption \ref{a2}, we derive that there exists a constant $L_{2}>0$ such that  
$$m_1\alpha(m^{(r)})+\frac{\check{p}-1}{2}\left| \beta(m_1)\right|^2 
\leq L_{2}(1+\sum_{i=1}^{r}|m_i|^2),$$
for any $m_i\in \mathbb{R}$, $i\in\mathbb{S}_r$.
\begin{ass}\label{a4}
	There exist constants $L_3>0$, $\sigma\in[\frac{1}{2},1)$ and $l_{2}, l_{3}\geq1$ such that
	$$\left|\alpha_{2}(m_{1})-\alpha_{2}(n_{1})\right|	\leq L_{3}\left(1+|m_1|^{l_{2}}+|n_{1}|^{l_{2}}\right)|m_1-n_{1}|,$$
	$$\left|\beta(m_1)-\beta(n_{1})\right|\leq L_{3}\left(1+|m_1|^{l_3}+|n_{1}|^{l_3}\right)|m_1-n_{1}|^\sigma,$$
	for any $m_i, n_i\in \mathbb{R}$, $i\in\mathbb{S}_r$.
\end{ass}

\begin{ass}\label{a3}
	There exist constants $L_4>0$ and $\gamma\in(0,1]$ such that
	\begin{equation*}
		\left|\xi(t)-\xi(s)\right|\leq L_{4}\left|t-s\right|^\gamma,
	\end{equation*}
	for any $t, s\in[-\tau,0]$.
\end{ass}

Under the given assumptions above, MDSDE (\ref{msdech}) admits a unique solution, which can be proved by consulting \cite{9}, \cite{10} and \cite{16}. The result is stated as the following lemma.

\begin{lem}\label{lem3.1}
	Let Assumptions \ref{a1}-\ref{a4} hold. Then MDSDE (\ref{msdech}) admits a unique solution z(t). Moreover, for any $p\in(0,\check{p}]$,
	$$\sup_{0\leq t\leq T}\mathbb{E}|z(t)|^p\leq C,~~~~\forall T>0.$$
\end{lem}

\begin{lem}\label{eseb}
	Let Assumptions \ref{a1}-\ref{a4} hold. Then for any $\bar{p}\in\left(0,\frac{{p}}{l_{3}+1}\right]$, the solution $z(t)$ of MDSDE (\ref{msdech}) satisfies
	$$\mathbb{E}\sup_{0\leq t\leq T}|z(t)|^{\bar{p}}\leq C,~~~~\forall T>0.$$
\end{lem}

\begin{proof}
	For any $0\leq t\leq T$ and $\bar{p}\geq 2$, using It\^o's formula, Young's inequality, Burkhold-Davis-Gundy's (BDG) inequality and H\"older's inequality gives that
	\begin{equation*}
		\begin{split}
			\mathbb{E}\sup_{0\leq s\leq t}|z(s)|^{\bar{p}}
			\leq&\mathbb{E}\Big[||\xi||^{\bar{p}} +\sup_{0\leq s\leq t}\bar{p}\int_{0}^{s}|z(u)|^{\bar{p}-2}[z(u)\alpha(\varPhi(z_u))+\frac{\bar{p}-1}{2}|\beta(z(u))|^2]du\Big]\\
			&+\bar{p}\mathbb{E}\sup_{0\leq s\leq t}\int_{0}^{s}|z(u)|^{\bar{p}-2} z(u)\beta(z(u))dB(u)\\
			\leq&\mathbb{E}||\xi||^{\bar{p}}+C\mathbb{E}\int_{0}^{t}\big(1+\sum_{v=1}^{r}|z(u+\bar{s}_v)|^{\bar{p}}\big)du
			+C\mathbb{E}\big(\int_{0}^{t}|z(u)|^{2\bar{p}-2}|\beta(z(u))|^2du\big)^\frac{1}{2}\\
			\leq&C+C\int_{0}^{t}\mathbb{E}\sup_{0\leq u\leq s}|z(u)|^{\bar{p}}  du
			+\frac{1}{2}\mathbb{E}\sup_{0\leq s\leq t}|z(s)|^{\bar{p}}+C\int_{0}^{t}\sup_{0\leq u\leq s}\mathbb{E}|z(u)|^{\bar{p}(l_3+1)}ds.
		\end{split}
	\end{equation*}
	The desired result follows by Gronwall's inequality, the condition $\bar{p}({l_{3}+1})\leq p$ and H\"older's inequality.
\end{proof}

Let us review the Yamada-Watanabe approximation technique in \cite{xin1, xin2}.
For any  $\Theta>1$ and $\varepsilon\in (0,1)$, we see $\int_{\frac{\varepsilon}{\Theta}}^{\varepsilon}\frac{1}{x}dx=\ln\Theta.$
For each $\Theta>1$ and $\varepsilon\in (0,1)$, there is a continuous function $\varPsi: \mathbb{R}^{+}\rightarrow\mathbb{R}^{+}$ with support $[\frac{\varepsilon}{\Theta},\varepsilon]$ such that 
$$0\leq\varPsi_{\Theta,\varepsilon}(x)\leq\frac{2}{x\ln\Theta},~~~
\int_{\frac{\varepsilon}{\Theta}}^{\varepsilon}\varPsi_{\Theta,\varepsilon}(x)dx=1.$$
Define $$U_{\Theta,\varepsilon}(x)=\int_{0}^{|x|}\int_{0}^{y}\varPsi_{\Theta,\varepsilon}(u)dudy.$$
For any  $x\in\mathbb{R}$, $U_{\Theta,\varepsilon}(x)$ has the following properties:
\begin{equation}
	U'_{\Theta,\varepsilon}(|x|)\leq\frac{\Theta}{\varepsilon}|x|,
\end{equation}
\begin{equation} \label{ym4.2}
	0\leq|U'_{\Theta,\varepsilon}(x)|\leq1,
\end{equation}
\begin{equation}\label{ym4.3}
	|x|\leq\varepsilon+U_{\Theta,\varepsilon}(x),
\end{equation}
\begin{equation}\label{ym4.4}
	0\leq U''_{\Theta,\varepsilon}(x)=\varPsi_{\Theta,\varepsilon}(|x|)\leq\frac{2}{|x|\cdot\ln\Theta}\mathbb{I}_{[\frac{\varepsilon}{\Theta},\varepsilon]}(|x|).
\end{equation}
Here, $U'$ and $U''$ are the first and the second order derivatives of $U$ w.r.t. $x$.

\section{The numerical scheme}
Now we review the TEMS which was originated in \cite{22,23}. The first step of defining the TEMS is to choose a strictly increasing continuous function $h: \mathbb{R}^{+}\rightarrow\mathbb{R}^{+}$ such that $h(w)\rightarrow\infty$ as $w\rightarrow\infty$ and
\begin{equation}
	\label{tru}\sup_{|m_{1}|\leq w}\left(|\alpha_{1}(m_{1})|\vee|\alpha_{2}(m_{1})|\vee|\beta(m_1)|\right)\leq h(w),~~~\forall w\geq1.
\end{equation}
Denote by  $h^{-1}$ the inverse function of $h$ and we see that $h^{-1}$ is a strictly increasing continuous function from $[h(1),\infty)$ to $[1,\infty)$. We also choose a strictly decreasing function $\Gamma:(0,1]\rightarrow(0,\infty)$ such that
\begin{equation}\label{truncated}
	\lim_{\Delta\rightarrow0}\Gamma(\Delta)=\infty,~~~\Delta^\frac{1}{4}\Gamma(\Delta)\leq L_0,~~~\forall\Delta\in(0,1].
\end{equation}

\begin{rem}
	The constraint on the step size $\Delta$ in ($\ref{truncated}$) is weaker than that in \cite{31,29}, so that in our paper the value interval of $\Delta$ is broader, which could reduce calculating amount. And we have more choices for the function $\Gamma(\Delta)$. Please refer to \cite{12} for more details.
\end{rem}

The truncated mapping $\varpi_{\Delta}: \mathbb{R}\rightarrow \mathbb{R}$ is defined by 
$$\varpi_{\Delta}(m_1)=\Big(|m_1|\wedge h^{-1}(\Gamma(\Delta))\Big)\frac{m_1}{|m_1|},$$ 
where set $\frac{m_1}{|m_1|}=0$ if $m_1=0$. Then denote 
$$\alpha_{1,\Delta}(m_{1})=\alpha_{1}(\varpi_{\Delta}(m_{1})),~\alpha_{2,\Delta}(m_{1})=\alpha_{2}(\varpi_{\Delta}(m_{1})),
~\beta_{\Delta}(m_1)=\beta(\varpi_{\Delta}(m_{1})).$$
Obviously, one can see that $|\alpha_{1,\Delta}(m_{1})|\vee|\alpha_{2,\Delta}(m_{1})|\vee|\beta_{\Delta}(m_1)|\leq \Gamma(\Delta).$
Moreover, let	$\alpha_{\Delta}(m^{(r)})=\alpha_{1,\Delta}(m_{1})+\alpha_{2,\Delta}(m_1)+\alpha_{3}(m^{(r)}).$

	\begin{rem}
		In reality, we use the partially TEMS in this paper. Instead of truncating the whole $\alpha$, we decompose $\alpha$ as $\alpha_1+\alpha_2+\alpha_3$, and just truncate  the superlinear parts (i.e., the linear part is reserved), which can improve the computational efficiency and preserve the	asymptotic stability of the underlying equations.
		Please refer to \cite{8} for more details . 
		
	\end{rem}

	Now, we introduce the TEMS for MDSDE. For two positive integers $M,M_T$ and a positive integer sequence $\{k_{v}\}$, $v\in\{2,3,\cdots,r-1\}$, define $\Delta=\tau_{v}/k_{v}=T/M_{T}=\tau/M$. Define $t_{k}=k\Delta$ for $k\in\big\{-M,\cdots, M_{T}\big\}$. Set $\chi_{\Delta}(t_{k})=\xi(t_{k})$ for $k\in\big\{-M,\cdots,0\big\}$. Then we form
	$$\chi_{\Delta}(t_{k+1})=\chi_{\Delta}(t_{k})+\alpha_{\Delta}(\varPhi(\chi_{\Delta,t_k}))\Delta+\beta_{\Delta}(\chi_{\Delta}(t_k))\Delta B_k,$$
	where $\varPhi(\chi_{\Delta,t_k}):=(\chi_{\Delta}(t_k),\chi_{\Delta}(t_{k-k_2}),\cdots,\chi_{\Delta}(t_{k-M}))$, and $\Delta B_k:=B(t_{k+1})-B(t_{k})$. Define the step process which is one of the two versions of the continuous-time truncated EM solutions:
	$$\bar{\chi}(t)=\sum_{k=-M}^{M_{T}}\chi_{\Delta}(t_k)\mathbb{I}_{[t_k,t_{k+1})}.$$
	The other continuous form is:
	\begin{equation}\label{continuous}
		\chi(t)=\chi(0)+\int_{0}^{t}\alpha_{\Delta}(\varPhi(\bar{\chi}_s))ds+\int_{0}^{t}\beta_{\Delta}(\bar{\chi}(s))dB(s),
	\end{equation}
	where $\varPhi(\bar{\chi}_s):=(\bar{\chi}(s),\bar{\chi}(s-\tau_{2}),\cdots,\bar{\chi}(s-\tau)).$
	It can be observed that, for any $t\in[t_k,t_{k+1})$, $\chi(t_k)=\chi_{\Delta}(t_k)=\bar{\chi}(t).$

	
	The following estimation is obtained by consulting Lemma 2.4 in \cite{9} and handling the multiple delay terms.
	\begin{lem}\label{lem2.2}
		Let Assumption \ref{a2} hold. For any $\Delta\in(0,1]$ and $m_i\in\mathbb{R}$, $i\in\mathbb{S}_r$, we derive that
				$$m_1\alpha_{\Delta}(m^{(r)})+\frac{\check{p}-1}{2}|\beta_{\Delta}(m_1)|^2
				\leq \frac{5}{2}L_{2}(1+|m_1|^2)+\frac{L_{2}(r-1)}{2}\sum_{i=2}^{r}|m_i|^2.$$

	\end{lem}

	\begin{lem}\label{lem3.3}
		Let Assumptions \ref{a1}-\ref{a4} hold. Then for any $\Delta\in(0,1]$ and $p\in(0,\check{p}]$, 
		$$\sup_{0<\Delta\leq1}\sup_{0\leq t\leq T}\mathbb{E}|\chi(t)|^p\leq C,$$
		and
		$$\mathbb{E}|\chi(t)-\bar{\chi}(t)|^p\leq C\Delta^\frac{p}{2}\Gamma^p(\Delta).$$
	\end{lem}
	
	\begin{proof}
		For any $t\in[0,T]$, 
		we get from It\^o's formula, Young's inequality, H\"older's inequality that
		\begin{equation*}	
			\begin{split}	
				\sup_{0\leq s\leq t}\mathbb{E}|\chi(s)|^p
				\leq&||\xi||^p+\sup_{0\leq s\leq t}\mathbb{E}\Big[p\int_{0}^{s}|\chi (u)|^{p-2}[\chi(u)\alpha_{\Delta}(\varPhi(\bar{\chi}_u))+\frac{p-1}{2}|\beta_{\Delta}(\bar{\chi}(u))|^2]du\\
				&+p\int_{0}^{s}|\chi(u)|^{p-2} \chi(u)\beta_{\Delta}(\bar{\chi}(u))dB(u)\Big]\\
				\leq&||\xi||^p+\sup_{0\leq s\leq t}\mathbb{E}\Big(p\int_{0}^{s}|\chi (u)|^{p-2}\big[\bar{\chi}(u)\alpha_{\Delta}(\varPhi(\bar{\chi}_u))+\frac{p-1}{2}|\beta_{\Delta}(\bar{\chi}(u))|^2\big]du\Big)\\
				&+p\mathbb{E}\int_{0}^{t}|\chi(s)|^{p-2}|\chi (s)-\bar{\chi}(s)||\alpha_{\Delta}(\varPhi(\bar{\chi}_s))|ds\\
				\leq&||\xi||^p+p\mathbb{E}\int_{0}^{t}|\chi(s)|^{p-2}\Big(\frac{5}{2}L_{2}(1+|\bar{\chi}(s)|^2)
				+\frac{L_{2}(r-1)}{2}\sum_{v=2}^{r}|\bar{\chi}(s+\bar{s}_v)|^2\Big)ds\\
				&+p\Gamma(\Delta)\mathbb{E}\int_{0}^{t}
				|\chi(s)|^{p-2}|\chi (s)-\bar{\chi}(s)|ds\\
				\leq&||\xi||^p+C\mathbb{E}\int_{0}^{t}\Big(1+|\chi(s)|^p+\sum_{v=1}^{r}|\bar{\chi}(s+\bar{s}_v)|^p\Big)ds
				+p\Gamma^\frac{p}{2}(\Delta)\int_{0}^{t}
				\mathbb{E}|\chi (s)-\bar{\chi}(s)|^\frac{p}{2}ds.
			\end{split}	
		\end{equation*}	
		The key point is to estimate $\mathbb{E}
		|\chi (t)-\bar{\chi}(t)|^\frac{p}{2}$.
		For any $t\in[0,T]$, 
		there always exists a positive integer $k$ such that $t\in[t_k,t_{k+1})$. From (\ref{continuous}), we get that
		$$\chi (t)-\bar{\chi}(t)=\int_{t_k}^{t}\alpha_{\Delta}(\varPhi(\chi_{\Delta,t_k}))ds+\int_{t_k}^{t}\beta_{\Delta}(\chi_{\Delta}(t_k))dB(s).$$
		For any $\Delta\in(0,1]$, $t\in[0,T]$ and $p\in[2,\check{p}]$, we derive that
		\begin{equation*}
			\begin{split}
				\mathbb{E}|\chi (t)-\bar{\chi}(t)|^\frac{p}{2}
				\leq&C\Delta^{\frac{p}{2}-1}\mathbb{E}\int_{t_k}^{t}|\alpha_{1,\Delta}(\chi_{\Delta}(t_k))|^\frac{p}{2}ds+C\Delta^{\frac{p}{2}-1}\mathbb{E}\int_{t_k}^{t}|\alpha_{2,\Delta}(\chi_{\Delta}(t_k))|^\frac{p}{2}ds\\&+C\Delta^{\frac{p}{2}-1}\mathbb{E}\int_{t_k}^{t}|\alpha_{3}(\varPhi(\chi_{\Delta,t_k}))|^\frac{p}{2}ds+C\Delta^{\frac{p}{4}-1}\mathbb{E}\int_{t_k}^{t}|\beta_{\Delta}(\chi_{\Delta}(t_k))|^\frac{p}{2}ds\\
				\leq&C\Delta^\frac{p}{4}\Gamma^\frac{p}{2}(\Delta)+C\Delta^\frac{p}{2}\big(1+\sup_{0\leq t\leq T}\mathbb{E}|\bar{\chi}(t)|^\frac{p}{2}+\|\xi\|^\frac{p}{2}\big).\\
			\end{split}
		\end{equation*}
	
		Therefore, combining these inequalities with Young's inequality yields that
		
		\begin{equation*}
			\begin{split}
				\sup_{0\leq s\leq t}\mathbb{E}|\chi(s)|^p
				\leq&C+C\int_{0}^{t}\sup_{0\leq u\leq s}\mathbb{E}|\chi(u)|^pds
				+C\int_{-\tau}^{0}\|\xi\|^p ds+p\Gamma^\frac{p}{2}(\Delta)\int_{0}^{t}
				\mathbb{E}|\chi (s)-\bar{\chi}(s)|^\frac{p}{2}ds\\
				\leq&C+C\int_{0}^{t}\sup_{0\leq u\leq s}\mathbb{E}|\chi(u)|^pds
				+C\Delta^\frac{p}{2}\Gamma^\frac{p}{2}(\Delta)\int_{0}^{t}
				\sup_{0\leq u\leq s}\mathbb{E}|\chi(u)|^\frac{p}{2}ds
				ds\\\leq&C+C\int_{0}^{t}\sup_{0\leq u\leq s}\mathbb{E}|\chi(u)|^pds.
			\end{split}
		\end{equation*}
		An application of the Gronwall inequality leads to the desired result for $p\in[2,\check{p}]$.
		The case when $p\in(0,2)$ holds with the aid of the Lyapunov inequality.
	\end{proof} 
	
	\begin{lem}\label{esnb}
		Let Assumptions \ref{a1}-\ref{a4} hold. Then for any $\Delta\in(0,1]$ and $\bar{p}\in\left(0,\frac{p}{l_{3}+1}\right]$, we have
		$$\sup_{\Delta\in(0,1]}\mathbb{E}\Big(\sup_{0\leq t\leq T}|\chi(t)|^{\bar{p}}\Big)\leq C.$$
	\end{lem}
	\begin{proof}
		By employing It\^o's formula, H\"older's inequality and Young's inequality, for $\bar{p}\geq 2$, we have
		\begin{equation*}
			\begin{split}
				\mathbb{E}\sup_{0\leq s\leq t}|\chi(t)|^{\bar{p}}
				\leq&\mathbb{E}\Big[||\xi||^{\bar{p}} +\sup_{0\leq s\leq t}\bar{p}\int_{0}^{s}|\chi(u)|^{\bar{p}-2}[\chi (u)\alpha_{\Delta}(\varPhi(\bar{\chi}_u))+\frac{\bar{p}-1}{2}|\beta_{\Delta}(\bar{\chi}(u))|^2]du\Big]\\
				&+\bar{p}\mathbb{E}\sup_{0\leq s\leq t}\int_{0}^{s}|\chi(u)|^{\bar{p}-2} \chi(u)\beta_{\Delta}(\bar{\chi}(u))dB(u)\\
				\leq&\mathbb{E}||\xi||^{\bar{p}}+C\mathbb{E}\int_{0}^{t}|\chi(u)|^{\bar{p}-2}\Big[\frac{5}{2}L_{2}(1+|\bar{\chi}(u)|^2)+\frac{L_{2}(r-1)}{2}\sum_{v=2}^{r}|\bar{\chi}(u+\bar{s}_v)|^2\Big]du\\
				&+\bar{p}\Gamma(\Delta)\mathbb{E}\int_{0}^{t}
				|\chi(u)|^{\bar{p}-2}|\chi (u)-\bar{\chi}(u)|du+C\mathbb{E}\big(\int_{0}^{t}|\chi(u)|^{2\bar{p}-2}|\beta_{\Delta}(\bar{\chi}(u))|^2du\big)^\frac{1}{2}\\
				\leq&C+C\int_{0}^{t}\mathbb{E}\sup_{0\leq u\leq s}|\chi(u)|^{\bar{p}}ds
				+\frac{1}{2}\mathbb{E}\sup_{0\leq s\leq t}|z(s)|^{\bar{p}}+C\int_{0}^{t}\sup_{0\leq u\leq s}\mathbb{E}|\chi(u)|^{\bar{p}(l_3+1)}ds.
			\end{split}
		\end{equation*}
		Due to Lemma \ref{lem3.3},  Gronwall's inequality, Lyapunov's inequality and the condition $\bar{p}(l_{3}+1)\leq p$, we get the desired result.
	\end{proof}
	
	The following lemma can be obtained immediately by borrowing the technique in Lemmas \ref{lem3.1} and \ref{lem3.3} with the Chebyshev inequality.
	\begin{lem}\label{lemma4.1}
		Let Assumptions \ref{a1}-\ref{a4} hold. For any $K>||\xi||$, define the stopping times:
		$$\rho_k=\inf\{t\in[0,T]:|z(t)|\geq K\},~\rho_{\Delta,k}=\inf\{t\in[0,T]:|\chi(t)|\geq K\}.$$
		Then we derive that
		$\mathbb{P}(\rho_k\leq T)\vee\mathbb{P}(\rho_{\Delta,k}\leq T)\leq\frac{C}{K^2}.$
	\end{lem}
	
	\section{Convergence rates at $T$ in $\mathcal{L}^{1}$ and $\mathcal{L}^{2}$ sense}
	
	
	\begin{lem}\label{lemma4.2}
		Let Assumptions \ref{a1}-\ref{a3} hold and assume that $K>||\xi||$, $p\geq\left[l_1\vee l_2\vee (2l_3)\right]$. Let $\Delta\in(0,1]$ be sufficiently small such that $h^{-1}(\Gamma(\Delta))\geq K$. Then for any $t\in[0,T]$, we have
		\begin{equation*}
			\begin{split}
				\mathbb{E}|z(t\wedge\rho_{*})-\chi(t\wedge\rho_{*})|
				&\leq \left\{\begin{array}{ll}
					C\Big(\frac{1}{\ln\Delta^{-1}}+\big(\Delta^{\frac{1}{2}}\Gamma(\Delta)\big)^{\eta}+\Delta^{\frac{3}{8}}\Gamma(\Delta)+\Delta^{\gamma}
					\Big), & \text { if } \sigma=\frac{1}{2}, \\			
					~~~\\
					C\Big(\big(\Delta^{\frac{1}{2}}\Gamma(\Delta)\big)^{{(2\sigma-1)}\wedge{\eta}}
					+\Delta^\gamma\Big), & \text { if } \sigma\in(\frac{1}{2},1), 
				\end{array}\right.
			\end{split}
		\end{equation*}
		where $\rho_{*}:=\rho_k\wedge\rho_{\Delta,k}$.
	\end{lem}

	\begin{proof}
		Let $\delta(t)=z(t)-\chi(t)$ for $t\in[0,T]$. For $0\leq s\leq t\wedge\rho_{*}$, we get
		$$|z(s+\bar{s}_v)|\vee|\chi(s+\bar{s}_v)|\leq K\leq h^{-1}(\Gamma(\Delta)),~~~\forall v\in\mathbb{S}_{r}.$$
		Thus, for $0\leq s\leq t\wedge\rho_{*}$, we see 
		$\alpha_{\Delta}(\varPhi(\bar{\chi}_s))=\alpha(\varPhi(\bar{\chi}_s)), \beta_{\Delta}(\varPhi(\bar{\chi}_s))=\beta(\varPhi(\bar{\chi}_s)).$
		By It\^o's formula and (\ref{ym4.2}), (\ref{ym4.3}), it gives that	
		\begin{equation}\label{JJ}
			\begin{split}
				&\mathbb{E}|\delta(t\wedge\rho_{*})|\\\leq&\varepsilon+\mathbb{E}U_{\Theta,\varepsilon}(\delta(t\wedge\rho_{*}))\\
				\leq&\varepsilon+\mathbb{E}\int_{0}^{t\wedge\rho_{*}}{U'}_{\Theta,\varepsilon}(\delta(s))\big[\alpha(\varPhi(z_s))-\alpha_{\Delta}(\varPhi(\bar{\chi}_s))\big]ds
				+\frac{1}{2}\mathbb{E}\int_{0}^{t\wedge\rho_{*}}{U''}_{\Theta,\varepsilon}(\delta(s))|\beta(z(s))-\beta_{\Delta}(\bar{\chi}(s))|^2 ds\\
				=:&\varepsilon+J_{1}+J_{2}.
			\end{split}
		\end{equation}
		Thus, we derive that 
		\begin{equation}\label{j1}
			\begin{split}
				J_{1}
				=&\mathbb{E}\int_{0}^{t\wedge\rho_{*}}{U'}_{\Theta,\varepsilon}(\delta(s))\big[\alpha(\varPhi(z_s))-\alpha(\varPhi(\bar{\chi}_s))\big]ds\\
				=&\mathbb{E}\int_{0}^{t\wedge\rho_{*}}{U'}_{\Theta,\varepsilon}(\delta(s))\big[\alpha_{1}(z(s))-\alpha_{1}(\bar{\chi}(s))\big]ds
				+\mathbb{E}\int_{0}^{t\wedge\rho_{*}}{U'}_{\Theta,\varepsilon}(\delta(s))\big[\alpha_{2}(z(s))-\alpha_{2}(\bar{\chi}(s))\big]ds\\
				&+\mathbb{E}\int_{0}^{t\wedge\rho_{*}}{U'}_{\Theta,\varepsilon}(\delta(s))\big[\alpha_{3}(\varPhi(z_s))-\alpha_{3}(\varPhi(\bar{\chi}_s))\big]ds\\
				=:&J_{11}+J_{12}+J_{13}.\\
			\end{split}
		\end{equation}
		By Assumption \ref{a1}, H\"older's inequality and (\ref{ym4.2}), we infer that
		\begin{equation}\label{j11}
			\begin{split}
				J_{11}=&\mathbb{E}\int_{0}^{t\wedge\rho_{*}}{U'}_{\Theta,\varepsilon}(\delta(s))\big[\alpha_{1}(z(s))-\alpha_{1}(\bar{\chi}(s))\big]ds\\
				\leq&\mathbb{E}\int_{0}^{t\wedge\rho_{*}}{U'}_{\Theta,\varepsilon}(\delta(s))\big[\alpha_{1}(z(s))-\alpha_{1}({\chi}(s))\big]ds
				+\mathbb{E}\int_{0}^{t\wedge\rho_{*}}{U'}_{\Theta,\varepsilon}(\delta(s))\big[\alpha_{1}(\chi(s))-\alpha_{1}(\bar{\chi}(s))\big]ds\\
				\leq&L_1\int_{0}^{t}\mathbb{E}\big(1+|\chi(s)|^{l_1}+|\bar{\chi}(s)|^{l_1}\big)|\chi(s)-\bar{\chi}(s)|^{\eta}ds\\
				\leq&L_{1}\int_{0}^{t}\big(1+\mathbb{E}|\chi(s)|^{p}+\mathbb{E}|\bar{\chi}(s)|^{p}\big)^{\frac{l_1}{p}}\big(\mathbb{E}|\chi(s)-\bar{\chi}(s)|^{\frac{\eta p}{p-l_1}}\big)^{\frac{p-l_1}{p}}ds\\
				\leq&C\big(\mathbb{E}|\chi(s)-\bar{\chi}(s)|^{\frac{\eta p}{p-l_1}}\big)^{\frac{p-l_1}{p}}
				\leq C(\Delta^\frac{1}{2}\Gamma(\Delta))^\eta.
			\end{split}
		\end{equation}
		where we utilize the fact that ${U'}_{\Theta,\varepsilon}(\delta(s))[\alpha_{1}(z(s))-\alpha_{1}({\chi}(s))]\leq 0$, since $${U'}_{\Theta,\varepsilon}(a-b)\geq0,~~~~~\alpha_{1}(a)-\alpha_{1}(b)\leq0,~~~~~a\geq b;$$
		$${U'}_{\Theta,\varepsilon}(a-b)<0,~~~~~\alpha_{1}(a)-\alpha_{1}(b)\geq0,~~~~~a< b.$$	
		We get from Assumptions \ref{a1}, \ref{a4} and (\ref{ym4.2}) that 
		\begin{equation}\label{j12}
			\begin{split}
				J_{12}=&\mathbb{E}\int_{0}^{t\wedge\rho_{*}}{U'}_{\Theta,\varepsilon}(\delta(s))\big[\alpha_{2}(z(s))-\alpha_{2}(\bar{\chi}(s))\big]ds\\
				=&\mathbb{E}\int_{0}^{t\wedge\rho_{*}}{U'}_{\Theta,\varepsilon}(\delta(s))\big[\alpha_{2}(z(s))-\alpha_{2}(\chi(s))\big]ds
				+\mathbb{E}\int_{0}^{t\wedge\rho_{*}}{U'}_{\Theta,\varepsilon}(\delta(s))\big[\alpha_{2}(\chi(s))-\alpha_{2}(\bar{\chi}(s))\big]ds\\
				\leq&\mathbb{E}\int_{0}^{t\wedge\rho_{*}}\frac{{U'}_{\Theta,\varepsilon}(\delta(s))}{z(s)-\chi(s)}\big(z(s)-\chi(s)\big)\big[\alpha_{2}(z(s))-\alpha_{2}(\chi(s))\big]ds\\
				&+L_{3}\mathbb{E}\int_{0}^{t\wedge\rho_{*}}\big(1+|\chi(s)|^{l_2}+|\bar{\chi}(s)|^{l_2}\big)|\chi(s)-\bar{\chi}(s)|ds\\
				\leq&L_{1}\int_{0}^{t}\mathbb{E}|z(s\wedge\rho_{*})-{\chi}(s\wedge\rho_{*})|ds+L_{3}\int_{0}^{t}\big(1+\mathbb{E}|\chi(s)|^{p}+\mathbb{E}|\bar{\chi}(s)|^{p}\big)^\frac{l_2}{p}\big(\mathbb{E}|\chi(s)-\bar{\chi}(s)|^\frac{p}{p-l_2}\big)^\frac{p-l_2}{p}ds\\
				\leq&C\int_{0}^{t}\mathbb{E}|\delta(s\wedge\rho_{*})|ds+C\Delta^\frac{1}{2}\Gamma(\Delta),
			\end{split}
		\end{equation}
		and
		\begin{equation}\label{j13}
			\begin{split}
				J_{13}=&\mathbb{E}\int_{0}^{t\wedge\rho_{*}}{U'}_{\Theta,\varepsilon}(\delta(s))\big[\alpha_{3}(\varPhi(z_s))-\alpha_{3}(\varPhi(\bar{\chi}_s))\big]ds\\
				\leq&L_{1}\mathbb{E}\int_{0}^{t\wedge\rho_{*}}\sum_{v=1}^{r}|z(s+\bar{s}_v)-\bar{\chi}(s+\bar{s}_v)|ds\\
				\leq&rL_{1}\mathbb{E}\int_{0}^{t\wedge\rho_{*}}|z(s)-\bar{\chi}(s)|ds+(r-1)L_{1}\int_{-\tau}^{0}|\xi(s)-\xi(\lfloor\frac{s}{\Delta}\rfloor\Delta)|ds\\
				\leq&C\int_{0}^{t}\mathbb{E}|\delta(s\wedge\rho_{*})|ds+C\mathbb{E}\int_{0}^{t\wedge\rho_{*}}|\chi(s)-\bar{\chi}(s)|ds+C\Delta^{\gamma}\\
				\leq&C\Big(\int_{0}^{t}\mathbb{E}|\delta(s\wedge\rho_{*})|ds+\Delta^\frac{1}{2}\Gamma(\Delta)+\Delta^{\gamma}\Big).
			\end{split}
		\end{equation}
		Substituting (\ref{j11}), (\ref{j12}) and (\ref{j13}) into (\ref{j1}) gives
		\begin{equation}
			\begin{split}
				J_{1}\leq C\left((\Delta^\frac{1}{2}\Gamma(\Delta))^\eta+\int_{0}^{t}\mathbb{E}|\delta(s\wedge\rho_{*})|ds+{\Delta}^{\gamma}\right).
			\end{split}
		\end{equation}
		By (\ref{ym4.4}), Assumption \ref{a4}, Lemmas \ref{lem3.1} and \ref{lem3.3}, we have
		\begin{equation*}
			\begin{split}
				J_2
				=&\frac{1}{2}\mathbb{E}\int_{0}^{t\wedge\rho_{*}}{U''}_{\Theta,\varepsilon}(\delta(s))|\beta(z(s))-\beta(\bar{\chi}(s))|^2 ds\\
				\leq&L_{3}^2\mathbb{E}\int_{0}^{t\wedge\rho_{*}}\frac{1}{|\delta(s)|\ln\Theta}\mathbb{I}_{[\frac{\varepsilon}{\Theta},\varepsilon]}(|\delta(s)|) \big(1+|z(s)|^{2l_3}+|\chi(s)|^{2l_3}\big)|z(s)-\chi(s)|^{2\sigma}ds\\
				&+L_{3}^2\mathbb{E}\int_{0}^{t\wedge\rho_{*}}\frac{1}{|\delta(s)|\ln\Theta}\mathbb{I}_{[\frac{\varepsilon}{\Theta},\varepsilon]}(|\delta(s)|) \big(1+|\chi(s)|^{2l_3}+|\bar{\chi}(s)|^{2l_3}\big)|\chi(s)-\bar{\chi}(s)|^{2\sigma}ds\\
				\leq&L_{3}^2\mathbb{E}\int_{0}^{t\wedge\rho_{*}}\frac{\varepsilon^{2\sigma-1}}{\ln\Theta}\big(1+|z(s)|^{2l_3}+|\chi(s)|^{2l_3}\big)ds
				\\&+L_{3}^2\mathbb{E}\int_{0}^{t\wedge\rho_{*}}\frac{\Theta}{\varepsilon\ln\Theta}\big(1+|\chi(s)|^{2l_3}+|\bar{\chi}(s)|^{2l_3}\big)|\chi(s)-\bar{\chi}(s)|^{2\sigma} ds\\
				\leq&L_{3}^2\big(\frac{\varepsilon^{2\sigma-1}}{\ln\Theta}\big)\int_{0}^{t}\big(1+\mathbb{E}|z(s)|^{2l_3}+\mathbb{E}|\chi(s)|^{2l_3}\big)ds\\
				&+L_{3}^2\big(\frac{\Theta}{\varepsilon\ln\Theta}\big)\int_{0}^{t}\Big(1+\mathbb{E}|\chi(s)|^{p}+\mathbb{E}|\bar{\chi}(s)|^{p}\Big)^\frac{2l_3}{p}\Big(\mathbb{E}|\chi(s)-\bar{\chi}(s)|^\frac{2\sigma p}{p-2l_3} \Big)^\frac{p-2l_3}{p}ds\\
				\leq&C\big(\frac{\varepsilon^{2\sigma-1}}{\ln\Theta}\big)+C\big(\frac{\Theta}{\varepsilon\ln\Theta}\big)(\Delta^{\frac{1}{2}}\Gamma(\Delta))^{2\sigma}.
			\end{split}
		\end{equation*}
		Then combining the estimations of $J_1$ and $J_2$ with the Gronwall inequality yields that
		%

$$				\mathbb{E}|\delta(t\wedge\rho_{*})|
				\leq C\Big(\varepsilon+(\Delta^{\frac{1}{2}}\Gamma(\Delta))^{\eta}+\Delta^{\gamma}+\frac{{\varepsilon}^{2\sigma-1}}{\ln{\Theta}}+\frac{\Theta}{\varepsilon\ln{\Theta}}(\Delta^{\frac{1}{2}}\Gamma(\Delta))^{2\sigma}
				\Big).$$
		When $\sigma=\frac{1}{2}$, choosing $\varepsilon=\frac{1}{\ln\Delta^{-1}}$ and $\Theta=\Delta^{-\frac{1}{8}}$ gives that
	$$	\mathbb{E}|\delta(t\wedge\rho_{*})|
				\leq C\Big(\frac{1}{\ln\Delta^{-1}}+\big(\Delta^{\frac{1}{2}}\Gamma(\Delta)\big)^{\eta}+\Delta^{\frac{3}{8}}\Gamma(\Delta)+\Delta^{\gamma}
				\Big).$$
		When $\sigma\in(\frac{1}{2},1)$, by taking $\varepsilon=\Delta^\frac{1}{2}\Gamma(\Delta)$, $\Theta=2$, we have
		$$\mathbb{E}|\delta(t\wedge\rho_{*})|
		\leq C\Big(\big(\Delta^{\frac{1}{2}}\Gamma(\Delta)\big)^{{(2\sigma-1)}\wedge{\eta}}
		+\Delta^\gamma\Big).$$
	\end{proof}

	\begin{thm}\label{thm4.3}
		Let the assumptions in Lemma \ref{lemma4.2} hold. For every sufficiently small $\Delta\in(0,1]$, if  $\Gamma(\Delta)\geq h\Big(\big(\Delta^\frac{1}{2}\Gamma(\Delta)\big)^{-1}\Big)$, then 
		\begin{equation*}
			\begin{split}
				\mathbb{E}|z(T)-\chi(T)|\leq
				\left\{\begin{array}{ll}
					C\Big(\frac{1}{\ln\Delta^{-1}}+\big(\Delta^{\frac{1}{2}}\Gamma(\Delta)\big)^{\eta}+\Delta^{\frac{3}{8}}\Gamma(\Delta)+\Delta^{\gamma}
					\Big), & \text {if }  \sigma=\frac{1}{2}, \\			
					~~~\\	C\Big(\big(\Delta^{\frac{1}{2}}\Gamma(\Delta)\big)^{{(2\sigma-1)}\wedge{\eta}}
					+\Delta^\gamma\Big), & \text {if }  \sigma\in(\frac{1}{2},1).
				\end{array}\right.
			\end{split}
		\end{equation*}
		
	\end{thm}
	\begin{proof}
		For any $\lambda>0$, by Lemma \ref{lemma4.1} and the following Young's inequality, 
		$$ab\leq\frac{\lambda}{q}a^{q}+\frac{q-1}{q\lambda^\frac{1}{q-1}}b^\frac{q}{q-1},~~~~\forall q>1,$$
		we have
		\begin{equation*}
			\begin{split}
				\mathbb{E}|\delta(t)|=\mathbb{E}|\delta(t)\mathbb{I}_{\{\rho_{*}>T\}}|+\mathbb{E}|\delta(t)\mathbb{I}_{\{\rho_{*}\leq T\}}|
				\leq&\mathbb{E}|\delta(t)\mathbb{I}_{\{\rho_{*}>T\}}|+\frac{\lambda}{2}\mathbb{E}|\delta(t)|^2+\frac{1}{2\lambda}\mathbb{P}(\rho_{*}\leq T)\\					\leq&\mathbb{E}|\delta(t)\mathbb{I}_{\{\rho_{*}>T\}}|+\frac{C\lambda}{2}+\frac{C}{2\lambda K^2}.
			\end{split}
		\end{equation*}
		Note that $h^{-1}\big(\Gamma(\Delta)\big)\geq\big(\Delta^\frac{1}{2}\Gamma(\Delta)\big)^{-1}$, since $\Gamma(\Delta)\geq h\Big(\big(\Delta^\frac{1}{2}\Gamma(\Delta)\big)^{-1}\Big)$. So taking $\lambda=\Delta^\frac{1}{2}\Gamma(\Delta)$ and $K=\big(\Delta^\frac{1}{2}\Gamma(\Delta)\big)^{-1}$ leads to 
		$$	\mathbb{E}|\delta(t)|\leq\mathbb{E}|\delta(t)\mathbb{I}_{\{\rho_{*}>T\}}|+C\Delta^\frac{1}{2}\Gamma(\Delta).$$
		According to Lemmas \ref{lemma4.1} and \ref{lemma4.2}, we obtain the desired result.
	\end{proof}
	
	For any $\epsilon\in(0,\frac{1}{4}]$,  there exists a subinterval $[\Delta_{1},\Delta_{2}]\subset(0,1]$ such that  $\frac{1}{\ln{\Delta^{-1}}}\leq\Delta^{\frac{3}{8}-\epsilon}$ holds for $\Delta\in[\Delta_{1},\Delta_{2}]$. We will use this technique to give the convergence rate in a special case in the following theorem.
	\begin{thm}\label{LL11}
		Let Assumptions \ref{a1}-\ref{a3} hold. Suppose that $p\geq\big[(2l_1+2\eta)\vee (2l_2+2)\vee(4l_3+4\sigma)\vee \frac{2l_3+2\sigma-1+\varPi}{\varPi}\big]$ for  any $\varPi\in (0,1)$. Let $l*=\big[(2l_1+2\eta-p)\vee (2l_2+2-p)\vee (2l_3+2\sigma-p)\big]$. Then for any $\Delta\in(0,1]$, we derive that
		\begin{equation*}
			\begin{split}
				\mathbb{E}|z(T)-\chi(T)|^2\leq
				\left\{\begin{array}{ll}
					C\Big(\Xi_1(\Delta)+\Xi_2(\Delta)
					+\Delta^{-\frac{1}{4}}\Xi_3(\Delta)
					\Big), & \text {if }  \sigma=\frac{1}{2}, \\   
					~~~\\ C\Big(\Xi_4(\Delta)+\Xi_2(\Delta)
					+\big(\Delta^\frac{1}{2}\Gamma(\Delta)\big)^{-2}\Xi_3(\Delta)\Big), & \text {if }  \sigma\in(\frac{1}{2},1),
				\end{array}\right.
			\end{split}
		\end{equation*}
		where
		$$\Xi_1(\Delta):=\Big(\frac{1}{\ln\Delta^{-1}}+\big(\Delta^{\frac{1}{2}}\Gamma(\Delta)\big)^{\eta}+\Delta^{\frac{3}{8}}\Gamma(\Delta)+\Delta^{\gamma}\Big)^{1-\varPi},$$ $$\Xi_2(\Delta):=\Big(h^{-1}(\Gamma(\Delta))\Big)^{l*},$$ $$\Xi_3(\Delta):=\Big(h^{-1}(\Gamma(\Delta))\Big)^{4l_3+4\sigma-p},$$
		$$\Xi_4(\Delta):=\Big(\big(\Delta^{\frac{1}{2}}\Gamma(\Delta)\big)^{({2\sigma-1})\wedge{\eta}}
		+\Delta^\gamma\Big)^{1-\varPi}.
		$$
		In particular, define
		$$h(w)=Lw^{\upsilon}, ~w>1,~~\text{and}~~\Gamma(\Delta)=\Delta^{-\epsilon},~\epsilon\in(0,\frac{1}{4}].$$
		Here, $L=L_{1}+2L_{3}+|\alpha_{1}(0)|+|\alpha_{2}(0)|+|\beta(0)|$, and $\upsilon=(l_1\vee l_2\vee l_3)+1.$ 
		
		Then, when $\sigma=\frac{1}{2}$, we have
		$$\mathbb{E}|z(T)-\chi(T)|^{2}\leq C\Delta^{\vartheta_{1}^{*}},~~~~~\Delta\in[\Delta_{1},\Delta_{2}],$$
		where $\vartheta_{1}^{*}=\big\{(1-\varPi)\big[(\frac{\eta}{2}-\eta\epsilon)\wedge(\frac{3}{8}-\epsilon)\wedge\gamma\big]\big\}\wedge\big\{-\frac{\epsilon l*}{\upsilon}\big\}\wedge\big\{\frac{\epsilon(p-4l_3-4\sigma)}{\upsilon}-\frac{1}{4}\big\}$.
		
		When $\sigma\in(\frac{1}{2},1)$, we have
		$$\mathbb{E}|z(T)-\chi(T)|^{2}\leq C\Delta^{\vartheta_{2}^{*}},~~~~~\Delta\in(0,1],$$
		where $\vartheta_{2}^{*}=\big\{(1-\varPi)\big[(\frac{\eta}{2}-\eta\epsilon)\wedge((\frac{1}{2}-\epsilon)(2\sigma-1))\wedge\gamma\big]\big\}\wedge\big\{-\frac{\epsilon l*}{\upsilon}\big\}\wedge\big\{\frac{\epsilon(p-4l_3-4\sigma)-\upsilon+2\upsilon\epsilon}{\upsilon}\big\}$.
	\end{thm}
	
	\begin{proof}
		For $t\in[0,T]$, let $\delta(t)=z(t)-\chi(t)$. For every integer $n>\|\xi\|$, define the stopping time
		$$\rho_{n}=\inf\{t\in[0,T]: |z(t)|\vee|\chi(t)|\geq n\}.$$
		Using It\^o's formula leads to
		\begin{equation*}
			\begin{split}
				&U_{\Theta,\varepsilon}(\delta(t\wedge\rho_{n}))\\
				=&\int_{0}^{t\wedge\rho_{n}}{U'}_{\Theta,\varepsilon}(\delta(s))\big[\alpha(\varPhi(z_s))-\alpha_{\Delta}(\varPhi(\bar{\chi}_s))\big]ds
				+\frac{1}{2}\int_{0}^{t\wedge\rho_{n}}{U''}_{\Theta,\varepsilon}(\delta(s))|\beta(z(s))-\beta_{\Delta}(\bar{\chi}(s))|^2 ds\\
				&+\int_{0}^{t\wedge\rho_{n}}{U'}_{\Theta,\varepsilon}(\delta(s))[\beta(z(s))-\beta_{\Delta}(\bar{\chi}(s))]dB(s)\\
				=:&S_{1}+S_{2}+S_{3}.
			\end{split}
		\end{equation*}
		We see that 
		\begin{equation*}
			\begin{split}
				S_1=&\int_{0}^{t\wedge\rho_{n}}{U'}_{\Theta,\varepsilon}(\delta(s))\big[\alpha(\varPhi(z_s))-\alpha_{\Delta}(\varPhi(\bar{\chi}_s))\big]ds\\
				=&\int_{0}^{t\wedge\rho_{n}}{U'}_{\Theta,\varepsilon}(\delta(s))\big[\alpha(\varPhi(z_s))-\alpha(\varPhi({\chi}_s))\big]ds
				+\int_{0}^{t\wedge\rho_{n}}{U'}_{\Theta,\varepsilon}(\delta(s))\big[\alpha(\varPhi(\chi_s))-\alpha_{\Delta}(\varPhi({\chi}_s))\big]ds\\
				&+\int_{0}^{t\wedge\rho_{n}}{U'}_{\Theta,\varepsilon}(\delta(s))\big[\alpha_{\Delta}(\varPhi(\chi_s))-\alpha_{\Delta}(\varPhi(\bar{\chi}_s))\big]ds\\
				=:&S_{11}+S_{12}+S_{13}.
			\end{split}
		\end{equation*}
		Similar to the estimation of $J_1$ in Lemma \ref{lemma4.2}, we get from Assumptions \ref{a1}, \ref{a4} and (\ref{ym4.2}) that
		\begin{equation*}
			\begin{split}
				S_{11}=&\int_{0}^{t\wedge\rho_{n}}{U'}_{\Theta,\varepsilon}(\delta(s))\big[\alpha(\varPhi(z_s))-\alpha(\varPhi({\chi}_s))\big]ds\\
				=&	\int_{0}^{t\wedge\rho_{*}}{U'}_{\Theta,\varepsilon}(\delta(s))\big[\alpha_{1}(z(s))-\alpha_{1}(\chi(s))\big]ds
				\\&+\int_{0}^{t\wedge\rho_{*}}{U'}_{\Theta,\varepsilon}(\delta(s))\big[\alpha_{2}(z(s))-\alpha_{2}(\chi(s))\big]ds\\
				&+\int_{0}^{t\wedge\rho_{n}}{U'}_{\Theta,\varepsilon}(\delta(s))\big[\alpha_{3}(\varPhi(z_s))-\alpha_{3}(\varPhi({\chi}_s))\big]ds\\
				\leq&L_1\int_{0}^{t\wedge\rho_{n}}|z(s)-\chi(s)|ds
				+L_{1}\int_{0}^{t\wedge\rho_{n}}\sum_{v=1}^{r}|z(s+\bar{s}_v)-\chi(s+\bar{s}_v)|ds\\
				\leq&C\int_{0}^{t\wedge\rho_{n}}|z(s)-\chi(s)|ds,
			\end{split}
		\end{equation*}
		\begin{equation*}
			\begin{split}
				S_{12}=&\int_{0}^{t\wedge\rho_{n}}{U'}_{\Theta,\varepsilon}(\delta(s))\big[\alpha(\varPhi(\chi_s))-\alpha_{\Delta}(\varPhi({\chi}_s))\big]ds\\
				\leq&\int_{0}^{t\wedge\rho_{n}}{U'}_{\Theta,\varepsilon}(\delta(s))\big[\alpha_{1}(\chi(s))-\alpha_{1}(\varpi_{\Delta}(\chi(s)))\big]ds
				\\&+\int_{0}^{t\wedge\rho_{n}}{U'}_{\Theta,\varepsilon}(\delta(s))\big[\alpha_{2}(\chi(s))-\alpha_{2}(\varpi_{\Delta}(\chi(s)))\big]ds\\
				\leq&L_{1}\int_{0}^{t\wedge\rho_{n}}\big(1+|\chi(s)|^{l_{1}}+|\varpi_{\Delta}(\chi(s))|^{l_{1}}\big)|\chi(s)-\varpi_{\Delta}(\chi(s))|^{\eta} ds\\&+L_{3}\int_{0}^{t\wedge\rho_{n}}\big(1+|\chi(s)|^{l_{2}}+|\varpi_{\Delta}(\chi(s))|^{l_{2}}\big)|\chi(s)-\varpi_{\Delta}(\chi(s))|
				,
			\end{split}
		\end{equation*}
		\begin{equation*}
			\begin{split}
				S_{13}=&
				\int_{0}^{t\wedge\rho_{n}}{U'}_{\Theta,\varepsilon}(\delta(s))\big[\alpha_{\Delta}(\varPhi(\chi_s))-\alpha_{\Delta}(\varPhi(\bar{\chi}_s))\big]ds\\
				=&\int_{0}^{t\wedge\rho_{*}}{U'}_{\Theta,\varepsilon}(\delta(s))\big[\alpha_{1,\Delta}(\chi(s))-\alpha_{1,\Delta}(\bar{\chi}(s))\big]ds
				\\&+\int_{0}^{t\wedge\rho_{*}}{U'}_{\Theta,\varepsilon}(\delta(s))\big[\alpha_{2,\Delta}(\chi(s))-\alpha_{2,\Delta}(\bar{\chi}(s))\big]ds\\
				&+\int_{0}^{t\wedge\rho_{*}}{U'}_{\Theta,\varepsilon}(\delta(s))\big[\alpha_{3}(\varPhi(\chi_s))-\alpha_{3}(\varPhi(\bar{\chi}_s))\big]ds\\
				\leq&L_{1}\int_{0}^{t\wedge\rho_{n}}\big(1+|\chi(s)|^{l_{1}}+|\bar{\chi}(s)|^{l_{1}}\big)|\chi(s)-\bar{\chi}(s)|^{\eta}ds\\&+L_{3}\int_{0}^{t\wedge\rho_{n}}\big(1+|\chi(s)|^{l_{2}}+|\bar{\chi}(s)|^{l_{2}}\big)|\chi(s)-\bar{\chi}(s)|ds\\
				&+L_{1}\int_{0}^{t\wedge\rho_{n}}\sum_{v=1}^{r}|\chi(s+\bar{s}_v)-\bar{\chi}(s+\bar{s}_v)|ds\\
				\leq&C\int_{0}^{t\wedge\rho_{n}}\big(1+|\chi(s)|^{l_{1}}+|\bar{\chi}(s)|^{l_{1}}\big)|\chi(s)-\bar{\chi}(s)|^{\eta}ds\\&+C\int_{0}^{t\wedge\rho_{n}}\big(1+|\chi(s)|^{l_{2}}+|\bar{\chi}(s)|^{l_{2}}\big)|\chi(s)-\bar{\chi}(s)|ds\\
				&+C\int_{0}^{t\wedge\rho_{n}}|\chi(s)-\bar{\chi}(s)|ds+C\int_{-\tau}^{0}|\xi(s)-\xi(\lfloor\frac{s}{\Delta}\rfloor\Delta)|ds.
			\end{split}
		\end{equation*}
		Then by Assumption \ref{a4} and (\ref{ym4.4}), we have
		
		\begin{equation*}
			\begin{split}
				S_{2}\leq& \frac{1}{2}\int_{0}^{t\wedge\rho_{n}}{U''}_{\Theta,\varepsilon}(\delta(s))|\beta(z(s))-\beta_{\Delta}(\bar{\chi}(s))|^2 ds\\
				\leq&C\int_{0}^{t\wedge\rho_{n}}{U''}_{\Theta,\varepsilon}(\delta(s))|\beta(z(s))-\beta(\chi(s))|^2 ds
				+C\int_{0}^{t\wedge\rho_{n}}{U''}_{\Theta,\varepsilon}(\delta(s))|\beta(\chi(s))-\beta_{\Delta}(\chi(s))|^2 ds\\
				&+C\int_{0}^{t\wedge\rho_{n}}{U''}_{\Theta,\varepsilon}(\delta(s))|\beta_{\Delta}(\chi(s))-\beta_{\Delta}(\bar{\chi}(s))|^2 ds\\
				\leq&
				C\big(\frac{\varepsilon^{2\sigma-1}}{\ln\Theta}\big)\int_{0}^{t\wedge\rho_{n}}\big(1+|z(s)|^{2l_{3}}+|\chi(s)|^{2l_{3}}\big)ds\\
				&+C\big(\frac{\Theta}{\varepsilon\ln\Theta}\big)\int_{0}^{t\wedge\rho_{n}}\big(1+|\chi(s)|^{2l_{3}}+|\bar{\chi}(s)|^{2l_{3}}\big)|\chi(s)-\bar{\chi}(s)|^{2\sigma} ds
				\\
				&+C\big(\frac{\Theta}{\varepsilon\ln\Theta}\big)\int_{0}^{t\wedge\rho_{n}}\big(1+|\chi(s)|^{2l_{3}}+|\varpi_{\Delta}(\chi(s))|^{2l_{3}}\big)|\chi(s)-\varpi_{\Delta}(\chi(s))|^{2\sigma} ds.
			\end{split}
		\end{equation*}
		Thanks to the It\^o isometry and Assumption \ref{a4}, we see
		\begin{equation*}
			\begin{split}
				\mathbb{E}\big[S_{3}\big]^2=&\mathbb{E}\Big[	\int_{0}^{t\wedge\rho_{n}}{U'}_{\Theta,\varepsilon}(\delta(s))[\beta(z(s))-\beta_{\Delta}(\bar{\chi}(s))]dB(s) \Big]^2\\
				=&\mathbb{E}	\int_{0}^{t\wedge\rho_{n}}|{U'}_{\Theta,\varepsilon}(\delta(s))|^2\big|\beta(z(s))-\beta_{\Delta}(\bar{\chi}(s))\big|^2ds\\ 
				\leq&
				C\mathbb{E}\int_{0}^{t\wedge\rho_{n}}\big(1+|z(s)|^{2l_{3}}+|\chi(s)|^{2l_{3}}\big)|z(s)-\chi(s)|^{2\sigma}ds\\&+C\mathbb{E}\int_{0}^{t\wedge\rho_{n}}\big(1+|\chi(s)|^{2l_{3}}+|\bar{\chi}(s)|^{2l_{3}}\big)|\chi(s)-\bar{\chi}(s)|^{2\sigma} ds
				\\
				&+C\mathbb{E}\int_{0}^{t\wedge\rho_{n}}\big(1+|\chi(s)|^{2l_{3}}+|\varpi_{\Delta}(\chi(s))|^{2l_{3}}\big)|\chi(s)-\varpi_{\Delta}(\chi(s))|^{2\sigma} ds.\\
			\end{split}
		\end{equation*}
		Hence, combining these inequalities with H\"older's inequality means that
		\begin{equation*}
			\begin{split}
				&\mathbb{E}|\delta(t\wedge\rho_{n})|^2
				\leq2\big[\varepsilon^2+\mathbb{E}U^2_{\Theta,\varepsilon}(\delta(t\wedge\rho_{n}))\big]\\
				\leq&C\big[\varepsilon^2+\int_{0}^{t}\mathbb{E}|\delta(s\wedge\rho_{n})|^2ds
				+\int_{0}^{t}\mathbb{E}\big(1+|\chi(s)|^{2l_{1}}+|\varpi_{\Delta}(\chi(s))|^{2l_{1}}\big)|\chi(s)-\varpi_{\Delta}(\chi(s))|^{2\eta} ds\\
				&+\int_{0}^{t}\mathbb{E}\big(1+|\chi(s)|^{2l_{2}}+|\varpi_{\Delta}(\chi(s))|^{2l_{2}}\big)|\chi(s)-\varpi_{\Delta}(\chi(s))|^2ds\\
				&+\int_{0}^{t}\mathbb{E}\big(1+|\chi(s)|^{2l_{1}}+|\bar{\chi}(s)|^{2l_{1}}\big)|\chi(s)-\bar{\chi}(s)|^{2\eta}ds\\&+\int_{0}^{t}\mathbb{E}\big(1+|\chi(s)|^{2l_{2}}+|\bar{\chi}(s)|^{2l_{2}}\big)|\chi(s)-\bar{\chi}(s)|^2ds+\int_{0}^{t}\mathbb{E}|\chi(s)-\bar{\chi}(s)|^2ds+\Delta^{2\gamma}\big]\\&
				+C\big(\frac{\varepsilon^{2\sigma-1}}{\ln\Theta}\big)^2\int_{0}^{t}\mathbb{E}\big(1+|z(s)|^{4l_{3}}+|\chi(s)|^{4l_{3}}\big)ds\\
				&+C\big(\frac{\Theta}{\varepsilon\ln\Theta}\big)^2\int_{0}^{t}\mathbb{E}\big(1+|\chi(s)|^{4l_{3}}+|\varpi_{\Delta}(\chi(s))|^{4l_{3}}\big)|\chi(s)-\varpi_{\Delta}(\chi(s))|^{4\sigma} ds\\
				&+C\big(\frac{\Theta}{\varepsilon\ln\Theta}\big)^2\int_{0}^{t}\mathbb{E}\big(1+|\chi(s)|^{4l_{3}}+|\bar{\chi}(s)|^{4l_{3}}\big)|\chi(s)-\bar{\chi}(s)|^{4\sigma} ds\\
				&	+C\mathbb{E}\int_{0}^{t\wedge\rho_{n}}\big(1+|z(s)|^{2l_{3}}+|\chi(s)|^{2l_{3}}\big)|z(s)-\chi(s)|^{2\sigma}ds\\
				&+C\int_{0}^{t}\mathbb{E}\big(1+|\chi(s)|^{2l_{3}}+|\varpi_{\Delta}(\chi(s))|^{2l_{3}}\big)|\chi(s)-\varpi_{\Delta}(\chi(s))|^{2\sigma} ds
				\\
				&+C\int_{0}^{t}\mathbb{E}\big(1+|\chi(s)|^{2l_{3}}+|\bar{\chi}(s)|^{2l_{3}}\big)|\chi(s)-\bar{\chi}(s)|^{2\sigma}ds.
			\end{split}
		\end{equation*}
		The key point is to estimate $\mathbb{E}\int_{0}^{t\wedge\rho_{n}}\big(1+|z(s)|^{2l_{3}}+|\chi(s)|^{2l_{3}}\big)|z(s)-\chi(s)|^{2\sigma}ds$.
		Let $\varPi\in(0,1)$ be arbitrary.
		Then the H\"older inequality leads to
		\begin{equation*}
			\begin{split}
				&	\mathbb{E}\int_{0}^{t\wedge\rho_{n}}\big(1+|z(s)|^{2l_{3}}+|\chi(s)|^{2l_{3}}\big)|z(s)-\chi(s)|^{2\sigma}ds\\
				\leq&	C\mathbb{E}\int_{0}^{t}\big(1+|z(s)|^{2l_{3}+2\sigma-1+\varPi}+|\chi(s)|^{2l_{3}+2\sigma-1+\varPi}\big)|z(s)-\chi(s)|^{1-\varPi}ds\\
				\leq&C\int_{0}^{t}\Big(\mathbb{E}\big(1+|z(s)|^\frac{2l_3+2\sigma-1+\varPi}{\varPi}+|\chi(s)|^\frac{2l_3+2\sigma-1+\varPi}{\varPi}\big)\Big)^{\varPi}\Big(\mathbb{E}|z(s)-\chi(s)|\Big)^{1-\varPi}ds\\
				\leq&C\int_{0}^{t}\Big(\mathbb{E}|z(s)-\chi(s)|\Big)^{1-\varPi}ds\\
				\leq&C\Big(\frac{1}{\ln\Delta^{-1}}+\big(\Delta^{\frac{1}{2}}\Gamma(\Delta)\big)^{\eta}+\Delta^{\frac{3}{8}}\Gamma(\Delta)+\Delta^{\gamma}\Big)^{1-\varPi}.
			\end{split}
		\end{equation*}
		By H\"older's inequality, Chebyshev's inequality and Lemmas \ref{lem3.1} and \ref{lem3.3}, we derive that
		\begin{equation*}
			\begin{split}
				&\int_{0}^{t}\mathbb{E}\big(1+|\chi(s)|^{4l_{3}}+|\varpi_{\Delta}(\chi(s))|^{4l_{3}}\big)|\chi(s)-\varpi_{\Delta}(\chi(s))|^{4\sigma} ds\\
				\leq&\int_{0}^{t}\Big(\mathbb{E}\big(1+|\chi(s)|^{p}+|\varpi_{\Delta}(\chi(s))|^{p}\big)\Big)^\frac{4l_{3}}{p}\Big(\mathbb{E}|\chi(s)-\varpi_{\Delta}(\chi(s))|^\frac{4\sigma p}{p-4l_{3}}\Big)^\frac{p-4l_{3}}{p}ds\\
				\leq&C\int_{0}^{t}\Big(\mathbb{E}\big[\mathbb{I}_{\{|\chi(s)|\geq h^{-1}(\Gamma(\Delta))\}}|\chi(s)|^\frac{4\sigma p}{p-4l_{3}}\big]\Big)^\frac{p-4l_{3}}{p}ds\\
				\leq&C\int_{0}^{t}\Big(\big(\mathbb{P}\{|\chi(s)|\geq h^{-1}(\Gamma(\Delta))\}\big)^\frac{p-4l_{3}-4\sigma}{p-4l_{3}}\big(\mathbb{E}|\chi(s)|^p\big)^\frac{4\sigma}{p-4l_{3}}\Big)^\frac{p-4l_{3}}{p}ds\\
				\leq&C\int_{0}^{t}\Big(\frac{\mathbb{E}|\chi(s)|^p}{(h^{-1}(\Gamma(\Delta)))^p}\Big)^\frac{p-4l_{3}-4\sigma}{p}ds\\
				\leq&C\Big(h^{-1}(\Gamma(\Delta))\Big)^{4l_{3}+4\sigma-p},
			\end{split}
		\end{equation*}
		\begin{equation*}
			\begin{split}
				\int_{0}^{t}\mathbb{E}\big(1+|\chi(s)|^{2l_{1}}+|\varpi_{\Delta}(\chi(s))|^{2l_{1}}\big)|\chi(s)-\varpi_{\Delta}(\chi(s))|^{2\eta} ds\leq C\Big(h^{-1}(\Gamma(\Delta))\Big)^{2l_{1}+2\eta-p},
			\end{split}
		\end{equation*} 
		\begin{equation*}
			\begin{split}
				\int_{0}^{t}\mathbb{E}\big(1+|\chi(s)|^{2l_{2}}+|\varpi_{\Delta}(\chi(s))|^{2l_{2}}\big)|\chi(s)-\varpi_{\Delta}(\chi(s))|^2ds\leq C\Big(h^{-1}(\Gamma(\Delta))\Big)^{2l_{2}+2-p},
			\end{split}
		\end{equation*}
		\begin{equation*}
			\begin{split}
				\int_{0}^{t}\mathbb{E}\big(1+|\chi(s)|^{2l_{3}}+|\varpi_{\Delta}(\chi(s))|^{2l_{3}}\big)|\chi(s)-\varpi_{\Delta}(\chi(s))|^{2\sigma} ds\leq C\Big(h^{-1}(\Gamma(\Delta))\Big)^{2l_{3}+2\sigma-p}.
			\end{split}
		\end{equation*}
				By H\"older's inequality and Lemma \ref{lem3.3}, it holds that
				\begin{equation*}
					\begin{split}
						&\int_{0}^{t}\mathbb{E}(1+|\chi(s)|^{2l_{1}}+|\bar{\chi}(s)|^{2l_{1}})|\chi(s)-\bar{\chi}(s)|^{2\eta}ds\\
						&\leq C\int_{0}^{t}\left(\mathbb{E}\big(1+|\chi(s)|^{4l_{1}}+|\bar{\chi}(s)|^{4l_{1}}\big)\right)^{\frac{1}{2}}\left(\mathbb{E}|\chi(s)-\bar{\chi}(s)|^{4\eta}\right)^{\frac{1}{2}}ds\leq C\big(\Delta^{\frac{1}{2}}\Gamma(\Delta)\big)^{{2\eta}},
					\end{split}
				\end{equation*}
				\begin{equation*}
					\begin{split}
						\int_{0}^{t}\mathbb{E}\big(1+|\chi(s)|^{2l_{2}}+|\bar{\chi}(s)|^{2l_{2}}\big)|\chi(s)-\bar{\chi}(s)|^2ds\leq C\big(\Delta^{\frac{1}{2}}\Gamma(\Delta)\big)^{2},
					\end{split}
				\end{equation*}
				\begin{equation*}
					\begin{split}
						\int_{0}^{t}\mathbb{E}|\chi(s)-\bar{\chi}(s)|^2ds\leq C\big(\Delta^{\frac{1}{2}}\Gamma(\Delta)\big)^{2},
					\end{split}
				\end{equation*}
				\begin{equation*}
					\begin{split}
						\int_{0}^{t}\mathbb{E}\big(1+|\chi(s)|^{4l_{3}}+|\bar{\chi}(s)|^{4l_{3}}\big)|\chi(s)-\bar{\chi}(s)|^{4\sigma} ds\leq C\big(\Delta^{\frac{1}{2}}\Gamma(\Delta)\big)^{4\sigma},
					\end{split}
				\end{equation*}
				\begin{equation*}
					\begin{split}
						\int_{0}^{t}\mathbb{E}\big(1+|\chi(s)|^{2l_{3}}+|\bar{\chi}(s)|^{2l_{3}}\big)|\chi(s)-\bar{\chi}(s)|^{2\sigma}ds\leq C\big(\Delta^{\frac{1}{2}}\Gamma(\Delta)\big)^{2\sigma}.
					\end{split}
				\end{equation*}
				By combining all these estimations with Gronwall's inequality,  we conclude that
				\begin{equation*}
					\begin{split}
						\mathbb{E}|\delta(t\wedge\rho_{n})|^2
						&\leq C\big[\varepsilon^2+\Big(h^{-1}(\Gamma(\Delta))\Big)^{2l_{1}+2\eta-p}+\Big(h^{-1}(\Gamma(\Delta))\Big)^{2l_2+2-p}+\big(\Delta^{\frac{1}{2}}\Gamma(\Delta)\big)^{{2\eta}\wedge{2\sigma}}+\Delta^{2\gamma}\big]					\\&
						+C\big(\frac{\varepsilon^{2\sigma-1}}{\ln\Theta}\big)^2
						+C\big(\frac{\Theta}{\varepsilon\ln\Theta}\big)^2\Big(h^{-1}(\Gamma(\Delta))\Big)^{4l_{3}+4\sigma-p}+C\big(\frac{\Theta}{\varepsilon\ln\Theta}\big)^2\big(\Delta^{\frac{1}{2}}\Gamma(\Delta)\big)^{4\sigma}\\
						&+C\Big(\frac{1}{\ln\Delta^{-1}}+\big(\Delta^{\frac{1}{2}}\Gamma(\Delta)\big)^{\eta}+\Delta^{\frac{3}{8}}\Gamma(\Delta)+\Delta^{\gamma}\Big)^{1-\varPi}
						+C\Big(h^{-1}(\Gamma(\Delta))\Big)^{2l_{3}+2\sigma-p}\\
						&\leq
						C\Big(\frac{1}{\ln\Delta^{-1}}+\big(\Delta^{\frac{1}{2}}\Gamma(\Delta)\big)^{\eta}+\Delta^{\frac{3}{8}}\Gamma(\Delta)+\Delta^{\gamma}\Big)^{1-\varPi}
						+C\Big(h^{-1}(\Gamma(\Delta))\Big)^{l*}\\
						&+
						C\left[\varepsilon^2
						+\big(\frac{\varepsilon^{2\sigma-1}}{\ln\Theta}\big)^2
						+\big(\frac{\Theta}{\varepsilon\ln\Theta}\big)^2\Big(h^{-1}(\Gamma(\Delta))\Big)^{4l_{3}+4\sigma-p}+\big(\frac{\Theta}{\varepsilon\ln\Theta}\big)^2\big(\Delta^{\frac{1}{2}}\Gamma(\Delta)\big)^{4\sigma}\right],
					\end{split}
				\end{equation*}
				where $l*=\left[(2l_1+2\eta-p)\vee (2l_2+2-p)\vee (2l_3+2\sigma-p)\right]$.
				
				If $\sigma=\frac{1}{2}$, choosing $\varepsilon=\frac{1}{\ln\Delta^{-1}}$ and $\Theta=\Delta^{-\frac{1}{8}}$ yields that 
				\begin{equation}\label{twoer2}
					\begin{split}
						\mathbb{E}\big|\delta(t\wedge\rho_{n})\big|^2\leq&C\Big(\frac{1}{\ln\Delta^{-1}}+\big(\Delta^{\frac{1}{2}}\Gamma(\Delta)\big)^{\eta}+\Delta^{\frac{3}{8}}\Gamma(\Delta)+\Delta^{\gamma}\Big)^{1-\varPi}
						\\&+C\Big(h^{-1}(\Gamma(\Delta))\Big)^{l*}
						+C\Delta^{-\frac{1}{4}}\Big(h^{-1}(\Gamma(\Delta))\Big)^{4l_{3}+4\sigma-p}.
					\end{split}
				\end{equation}
				
				If $\sigma\in(\frac{1}{2},1)$, choosing $\varepsilon=\Delta^\frac{1}{2}\Gamma(\Delta)$, $\Theta=2$ gives that 
				\begin{equation*}
					\begin{split}
						\mathbb{E}\big|\delta(t\wedge\rho_{n})\big|^2\leq&C\Big(\big(\Delta^{\frac{1}{2}}\Gamma(\Delta)\big)^{({2\sigma-1})\wedge{\eta}}
						+\Delta^\gamma\Big)^{1-\varPi}\\&+C\Big(h^{-1}(\Gamma(\Delta))\Big)^{l*}
						+C\big(\Delta^\frac{1}{2}\Gamma(\Delta)\big)^{-2}\Big(h^{-1}(\Gamma(\Delta))\Big)^{4l_{3}+4\sigma-p}.
					\end{split}
				\end{equation*}
				The desired  assertion follows from the Fatou lemma by letting $n\rightarrow\infty$.	
			\end{proof}
			\begin{rem}	
				When taking the value of $\varPi$ close to $0$, the optimal convergence rate in $\mathcal{L}^{2}$ sense is obtained.
			\end{rem}
			\begin{rem}
				We only show the convergence rates in $\mathcal{L}^{2}$ sense, but there in no doubt that we can get the results in $\mathcal{L}^{p}(p\geq 2)$ sense with the aid of the technique of Theorem 4.4 in \cite{29}. However, from Theorem 4.4 in \cite{29}, we see that the  convergence rate will not increase with the increase of $p$. Thus, we omit the case when $p>2$.	
			\end{rem}
			\section{Convergence rates over [0,T] in $\mathcal{L}^{1}$ and $\mathcal{L}^{2}$ sense}
			\begin{lem}\label{lemma4.4}
				Let Assumptions \ref{a1}-\ref{a3} hold. Suppose that $K>||\xi||$ and $\bar{p}\geq\left[l_1\vee l_2\vee (2l_3)\right]$. Let $\Delta\in(0,1]$ be sufficiently small such that $h^{-1}(\Gamma(\Delta))\geq K$. Then,
				\begin{equation*}
					\begin{split}
						\mathbb{E}\sup_{0\leq t\leq T}|z(t\wedge\rho_{*})-\chi(t\wedge\rho_{*})|
						&\leq \left\{\begin{array}{ll}
							C\Big(\frac{1}{\ln\Delta^{-1}}+\big(\Delta^{\frac{1}{2}}\Gamma(\Delta)\big)^{\eta}+\Delta^{\frac{3}{8}}\Gamma(\Delta)+\Delta^{\gamma}
							\Big)^\frac{1}{2}, & \text { if } \sigma=\frac{1}{2}, \\			
							~~~\\
							C\Big(\big(\Delta^{\frac{1}{2}}\Gamma(\Delta)\big)^{{(2\sigma-1)}\wedge{\eta}}
							+\Delta^\gamma\Big)^{2\sigma-1 }, & \text { if } \sigma\in(\frac{1}{2},1), 
						\end{array}\right.
					\end{split}
				\end{equation*}
				where $\rho_{*}:=\rho_k\wedge\rho_{\Delta,k}$.
			\end{lem}
			\begin{proof}
				By It\^o's formula, for any $t\in[0,T]$, we know that
				\begin{equation*}
					\begin{split}
						\mathbb{E}\sup_{0\leq s\leq t}|\delta(s\wedge\rho_{*})|
						\leq&\varepsilon+\mathbb{E}\sup_{0\leq s\leq t}\int_{0}^{s\wedge\rho_{*}}{U'}_{\Theta,\varepsilon}(\delta(u))\big[\alpha(\varPhi(z_u))-\alpha_{\Delta}(\varPhi(\bar{\chi}_u))\big]du\\
						&+\frac{1}{2}\mathbb{E}\sup_{0\leq s\leq t}\int_{0}^{s\wedge\rho_{*}}{U''}_{\Theta,\varepsilon}(\delta(u))|\beta(z(u))-\beta_{\Delta}(\bar{\chi}(u))|^2 du\\
						&+\mathbb{E}\sup_{0\leq s\leq t}\int_{0}^{s\wedge\rho_{*}}{U'}_{\Theta,\varepsilon}(\delta(u))[\beta(z(u))-\beta_{\Delta}(\bar{\chi}(u))]dB(u)\\
						=:&\varepsilon+D_{1}+D_{2}+D_{3}.
					\end{split}
				\end{equation*}
				In a similar way to the proof in Lemma \ref{lemma4.2}, it is easy to see that
				\begin{equation*}
					\begin{split}
						D_{1}=&
						\mathbb{E}\sup_{0\leq s\leq t}\int_{0}^{s\wedge\rho_{*}}{U'}_{\Theta,\varepsilon}(\delta(u))\big[\alpha(\varPhi(z_u))-\alpha(\varPhi(\bar{\chi}_u))\big]du\\
						\leq&C\Big[\int_{0}^{t}\mathbb{E}\sup_{0\leq u\leq s}|\delta(u\wedge\rho_{*})|ds+\big(\Delta^\frac{1}{2}\Gamma(\Delta)\big)^{\eta}+\Delta^{\gamma}\Big],
					\end{split}
				\end{equation*}
				\begin{equation*}
					\begin{split}
						D_{2}
						&=\frac{1}{2}\mathbb{E}\sup_{0\leq s\leq t}\int_{0}^{s\wedge\rho_{*}}{U''}_{\Theta,\varepsilon}(\delta(u))|\beta(z(u))-\beta(\bar{\chi}(u))|^2 du\\
						&\leq C\left[\frac{\varepsilon^{2\sigma-1}}{\ln\Theta}+\big(\frac{\Theta}{\varepsilon\ln\Theta}\big)\big(\Delta^\frac{1}{2}\Gamma(\Delta)\big)^{2\sigma}\right].
					\end{split}
				\end{equation*}
				By (\ref{ym4.2}), Assumption \ref{a4}, BDG's inequality and H\"older's inequality, we have
				\begin{equation*}
					\begin{split}
						D_{3}
						\leq& C\mathbb{E}\Big(\int_{0}^{t\wedge\rho_{*}}|\beta(z(s))-\beta(\bar{\chi}(s))|^{2}ds\Big)^\frac{1}{2}\\
						\leq&C\mathbb{E}\Big(\int_{0}^{t\wedge\rho_{*}}\big(1+|z(s)|^{2l_{3}}+|\chi(s)|^{2l_{3}}\big)|z(s)-\chi(s)|^{2\sigma}ds\Big)^\frac{1}{2}\\
						&+C\mathbb{E}\Big(\int_{0}^{t\wedge\rho_{*}}\big(1+|\chi(s)|^{2l_{3}}+|\bar{\chi}(s)|^{2l_{3}}\big)|\chi(s)-\bar{\chi}(s)|^{2\sigma}ds\Big)^\frac{1}{2}.\\
					\end{split}
				\end{equation*}
				
				When $\sigma=\frac{1}{2}$, by Young's inequality and Lemmas \ref{eseb}, \ref{esnb}, we see that
				\begin{equation*}
					\begin{split}
						D_{3}
						\leq&C\Big(1+\mathbb{E}\sup_{0\leq s\leq t}|z(s)|^{2l_3}+\mathbb{E}\sup_{0\leq s\leq t}|\chi(s)|^{2l_3}\Big)^\frac{1}{2}\Big(\int_{0}^{t}\mathbb{E}|z(s)-\chi(s)|ds\Big)^\frac{1}{2}\\
						&+C\Big(1+\mathbb{E}\sup_{0\leq s\leq t}|\chi(s)|^{2l_3}\Big)^\frac{1}{2}\Big(\int_{0}^{t}\mathbb{E}|\chi(s)-\bar{\chi}(s)|ds\Big)^\frac{1}{2}\\
						\leq&C\Big(\mathbb{E}|\delta(t)|\Big)^\frac{1}{2}+\Big(\Delta^\frac{1}{2}\Gamma(\Delta)\Big)^\frac{1}{2}.
					\end{split}
				\end{equation*}	
				By applying Gronwall's inequality  and choosing $\varepsilon=\frac{1}{\ln\Delta^{-1}}$, $\Theta=\Delta^{-\frac{1}{8}}$, we get from Theorem \ref{thm4.3} that
			$$	\mathbb{E}\sup_{0\leq s\leq t}|\delta(s\wedge\rho_{*})|
						\leq C\Big(\frac{1}{\ln\Delta^{-1}}+\big(\Delta^{\frac{1}{2}}\Gamma(\Delta)\big)^{\eta}+\Delta^{\frac{3}{8}}\Gamma(\Delta)+\Delta^{\gamma}
						\Big)^\frac{1}{2}.$$
				
				When $\sigma\in(\frac{1}{2},1)$, by Young's inequality, H\"older's inequality and Lemmas \ref{eseb}, \ref{esnb}, we derive that
				\begin{equation*}
					\begin{split}
						D_{3}
						\leq&C\mathbb{E}\Big[\big(\sup_{0\leq s\leq t}|\delta(s\wedge\rho_{*})|\big)^\frac{1}{2}\Big(\int_{0}^{t}\big(1+|z(s)|^{2l_{3}}+|\chi(s)|^{2l_{3}}\big)|z(s)-\chi(s)|^{2\sigma-1}ds\Big)^\frac{1}{2}\Big]\\
						&+C\mathbb{E}\Big(\int_{0}^{t}\big(1+|\chi(s)|^{2l_{3}}+|\bar{\chi}(s)|^{2l_{3}}\big)|\chi(s)-\bar{\chi}(s)|^{2\sigma}ds\Big)^\frac{1}{2}\\
						\leq&\frac{1}{2}\mathbb{E}\sup_{0\leq s\leq t}|\delta(s\wedge\rho_{*})|+C\mathbb{E}\int_{0}^{t}\big(1+|z(s)|^{2l_{3}}+|\chi(s)|^{2l_{3}}\big)|z(s)-\chi(s)|^{2\sigma-1}ds\\
						&+\Big(1+\mathbb{E}\sup_{0\leq s\leq t}|\chi(s)|^{2l_3}\Big)^\frac{1}{2}\Big(\int_{0}^{t}\mathbb{E}|\chi(s)-\bar{\chi}(s)|^{2\sigma}ds\Big)^\frac{1}{2}\\
						\leq&\frac{1}{2}\mathbb{E}\sup_{0\leq s\leq t}|\delta(s\wedge\rho_{*})|+C\int_{0}^{t}\big(\mathbb{E}|z(s)-\chi(s)|\big)^{2\sigma-1}ds+C\Big(\int_{0}^{t}\mathbb{E}|\chi(s)-\bar{\chi}(s)|^{2\sigma}ds\Big)^\frac{1}{2}.
					\end{split}
				\end{equation*}
				Due to Gronwall's inequality and Theorem \ref{thm4.3}, choosing $\varepsilon=\Delta^\frac{1}{2}\Gamma(\Delta)$, $\Theta=2$ means that
				\begin{equation*}
					\begin{split}
						\mathbb{E}\sup_{0\leq s\leq t}|\delta(s\wedge\rho_{*})|
						\leq C\Big(\big(\Delta^{\frac{1}{2}}\Gamma(\Delta)\big)^{{(2\sigma-1)}\wedge{\eta}}
						+\Delta^\gamma\Big)^{2\sigma-1}.
					\end{split}
				\end{equation*}
			\end{proof}

			\begin{thm}
				Let the assumptions in Lemma \ref{lemma4.4} hold. If  $\Gamma(\Delta)\geq h(\big(\Delta^\frac{1}{2}\Gamma(\Delta)\big)^{-1})$ holds for every sufficiently small $\Delta\in(0,1]$, then we have
				\begin{equation*}
					\begin{split}
						\mathbb{E}\sup_{0\leq t\leq T}|z(t)-\chi(t)|\leq \left\{\begin{array}{ll}
							C\Big(\frac{1}{\ln\Delta^{-1}}+\big(\Delta^{\frac{1}{2}}\Gamma(\Delta)\big)^{\eta}+\Delta^{\frac{3}{8}}\Gamma(\Delta)+\Delta^{\gamma}
							\Big)^\frac{1}{2}, & \text { if } \sigma=\frac{1}{2}, \\			
							~~~\\	C\Big(\big(\Delta^{\frac{1}{2}}\Gamma(\Delta)\big)^{{(2\sigma-1)}\wedge{\eta}}
							+\Delta^\gamma\Big)^{2\sigma-1 }, & \text { if } \sigma\in(\frac{1}{2},1).
						\end{array}\right.
					\end{split}
				\end{equation*}
				
			\end{thm}
			
			\begin{proof}
				For $\lambda>0$, using Young's inequality leads to
				\begin{equation*}
					\begin{split}
						\mathbb{E}\sup_{0\leq s\leq t}|\delta(s)|\leq&\mathbb{E}\sup_{0\leq s\leq t}|\delta(s)\mathbb{I}_{\{\rho_{*}>T\}}|+\frac{C\lambda}{2}+\frac{C}{2\lambda K^2}.
					\end{split}
				\end{equation*}
				We know 	$h^{-1}\big(\Gamma(\Delta)\big)\geq\big(\Delta^\frac{1}{2}\Gamma(\Delta)\big)^{-1}$,
				since $\Gamma(\Delta)\geq h\Big(\big(\Delta^\frac{1}{2}\Gamma(\Delta)\big)^{-1}\Big)$. Choosing $\lambda=\Delta^\frac{1}{2}\Gamma(\Delta)$, $K=\big(\Delta^\frac{1}{2}\Gamma(\Delta)\big)^{-1}$ yields that
				\begin{equation*}
					\begin{split}
						\mathbb{E}\sup_{0\leq s\leq t}|\delta(s)|\leq&\mathbb{E}\sup_{0\leq s\leq t}|\delta(s)\mathbb{I}_{\{\rho_{*}>T\}}|+\Delta^\frac{1}{2}\Gamma(\Delta).
					\end{split}
				\end{equation*}	
				Thereupon, the desired result can be obtained.
			\end{proof}
			
			\begin{thm}
				Let the assumptions in Theorem \ref{LL11} hold. Then		
				\begin{equation*}
					\begin{split}
						\mathbb{E}\sup_{0\leq t\leq T}|z(t)-\chi(t)|^2\leq
						\left\{\begin{array}{ll}
							C\Big(\Xi_1(\Delta)+\Xi_2(\Delta)
							+\Delta^{-\frac{1}{4}}\Xi_3(\Delta)
							\Big), & \text {if }  \sigma=\frac{1}{2}, \\   
							~~~\\ C\Big(\Xi_4(\Delta)+\Xi_2(\Delta)
							+\big(\Delta^\frac{1}{2}\Gamma(\Delta)\big)^{-2}\Xi_3(\Delta)\Big), & \text {if }  \sigma\in(\frac{1}{2},1),
						\end{array}\right.
					\end{split}
				\end{equation*}
				where the notations are all the same as those in Theorem \ref{LL11}.
			\end{thm}
			\begin{proof}	
				An application of BDG's inequality leads to
				\begin{equation*}
					\begin{split}
						&\mathbb{E}\sup_{0\leq s\leq t}\Big[ \int_{0}^{s\wedge\rho_{n}}{U'}_{\Theta,\varepsilon}(\delta(u))[\beta(z(u))-\beta_{\Delta}(\bar{\chi}(u))]dB(u) \Big]^2
						\leq C\mathbb{E} \int_{0}^{t\wedge\rho_{n}}|{U'}_{\Theta,\varepsilon}(\delta(s))|^2\big|\beta(z(s))-\beta_{\Delta}(\bar{\chi}(s))\big|^2ds.\\ 
					\end{split}
				\end{equation*}
				Then the proof process resembles the proof of Theorem \ref{LL11}, so we omit it. 	
			\end{proof}

			\begin{rem}
				There is no doubt that the four main conclusions in this paper can be generalized to the corresponding nonautonomous equations.
				For the processing technology of nonautonomous equations, please refer to \cite{xin3}.
			\end{rem}

			\begin{rem}\label{muldel}
				From the conclusions, we observe that the number of delay variables only exerts influence on the constant $C$ but not  impacts the convergence rates, which is checked in the numerical experiment in Section 6.
			\end{rem}

			\section{Stochastic volatility model and its numerical experiment}
			
			As is known to all, the stock prices always fluctuate, which will make an impact on the values of financial products. So investors constantly look for a financial instrument that reduces the volatility of stock prices. 
			Under this background, VIX  plays a key role, which  is a measure of overall market volatility. 
			From {\cite{60}}, we see that: when the market is volatile and the price fluctuates sharply, the VIX  tends to rise; when the market rises slowly in a long-term bull market, the VIX remains low and stable.
			There is a general  volatility model used to value the volatility of financial products and capture the dynamic changes of the VIX:
			$$dV(t)=\big(c_{1}+\frac{c_{2}}{V(t)}+c_{3}\frac{V(t)}{\ln V(t)}+c_{4}V(t)+c_{5}V^{2}(t)\big)dt+kV^{\kappa}(t)dB(t).$$
			Here, $V$ represents the volatility and $k, c_{j}$ are constants for $j=1,2,\cdots,5$. And the value of $\kappa$ is an important feature to distinguish different stochastic volatility models \cite{60}.
			From \cite{61}, by taking $\kappa=3/2, c_{1}={c_{2}}=c_{3}=0$, it becomes:
			$$dV(t)=\big(c_{4}V(t)+c_{5}V^{2}(t)\big)dt+kV^{\frac{3}{2}}(t)dB(t),$$
			which is one of the best models tested for describing the behavior of the VIX, and can explain the phenomenon of the finance theory in a simple way. Then the risk-neutral process follows:
			\begin{equation}\label{VIX}
				dV(t)=\big(c_{4}V(t)+c_{5}V^{2}(t)-\lambda^{*}kV^{\frac{3}{2}}(t)\big)dt+kV^{\frac{3}{2}}(t)dB(t).
			\end{equation}
			For more details about stochastic volatility models and their numerical schemes, please refer to \cite{xin27,60,61,xin26} and references therein.
			
			Next, as described in the introduction, we perform the TEMS for the  following scalar MDSDE (by choosing appropriate constants in (\ref{vv2})):
			\begin{equation}\label{vq2}
				\begin{split}
					dz(t)=&\big(-3z(t)|z(t)|-4z^{3}(t)+2z(t)+z(t-1)
					+3z(t-0.25)-z(t)|z(t)|^{\frac{1}{2}}\big)dt+|z(t)|^{\frac{3}{2}}dB(t),
				\end{split}
			\end{equation}
			with the initial data $\xi(t)=|t|^\frac{1}{2}+2$, $t\in[-1,0]$.
			Obviously, $\alpha_1(z_1)=-z_1|z_1|^{\frac{1}{2}}, \alpha_2(z_1)=-3z_1|z_1|-4z_1^{3},\alpha_3(z_1,z_2,z_3)=2z_1+3z_2+z_3$.
			One can easily see that Assumptions \ref{a1}-\ref{a3} hold. 
			We verify that the first inequality of Assumption \ref{a1} holds with $l_{1}=1$ and $\eta=\frac{1}{2}$, then we give the proof process of the second inequality of Assumption \ref{a1} in detail:
			\begin{equation*}
				\begin{split}
					&(z_1-y_1)(\alpha_2(z_1)-\alpha_2(y_1))\\
					&=(z_1-y_1)(-3z_1|z_1|-4z_1^{3}+3y_1|y_1|+4y_1^{3})\\
					&\leq -3(|z_1|+|y_1|)(|z_1|-|y_1|)^2-4|z_1-y_1|^2(z_1^2+z_1y_1+y_1^2)\\
					&\leq -4|z_1-y_1|^2|z_1^2+y_1^2|+4|z_1-y_1|^2|z_1y_1|\\
					&\leq |z_1-y_1|^2,
				\end{split}
			\end{equation*}
			where the inequality $(|a|+|b|)(|a|-|b|)^2\leq (a-b)(a|a|-b|b|), \forall a,b\in \mathbb{R}$ has been used.
			Define $h(w)=8w^3$ and $\Gamma(\Delta)=\Delta^{-\frac{1}{4}}$. The two functions meet the requirements (\ref{tru}) and (\ref{truncated}), and $h^{-1}\big(\Gamma(\Delta)\big)=\Delta^{-1/12}/2$.  Since the analytical solution is difficult to be expressed, we use the numerical solution with stepsize $\Delta=2^{-13}$ as the analytical
			solution. Then the $\mathcal{L}^1$-convergence rate can be computed by the mean error between the exact solution and numerical solution with stepsize $\Delta=2^{-12}, 2^{-11}, 2^{-10}, 2^{-9}$ at $T=2$ of 500 independent trajectories. The convergence rate, which is about 0.125, is shown in Figure \ref {tu1}.

Moreover, to check the conclusion in Remark \ref{muldel}, we perform the experiment for the following examples.
\begin{equation}\label{mul1}
	\begin{split}
		dz(t)=&\big(-4z^{3}(t)+z(t)+z(t-2^{-9})\big)dt+|z(t)|^{\frac{1}{2}}dB(t).
	\end{split}
\end{equation}
\begin{equation}\label{mul2}
	\begin{split}
		dz(t)=&\big(-4z^{3}(t)+z(t)+\sum_{j=1}^{256}z(t-j\cdot 2^{-9})\big)dt+|z(t)|^{\frac{1}{2}}dB(t).
	\end{split}
\end{equation}
\begin{equation}\label{mul3}
	\begin{split}
		dz(t)=&\big(-4z^{3}(t)+z(t)+\sum_{j=1}^{512}z(t-j\cdot 2^{-9})\big)dt+|z(t)|^{\frac{1}{2}}dB(t).
	\end{split}
\end{equation}
We see that there are 1 delay, 256 delays, 512 delays in (\ref{mul1}), (\ref{mul2}), (\ref{mul3}), respectively.
Note the result in \cite{xin29}:
\begin{equation}\label{mumuer}
	\begin{split}
		\log err_{\Delta}\approx \log C +p^* \log \Delta,
	\end{split}
\end{equation}
where $p^*$ is the convergence rate.

From Figure \ref{tu2}, we can observe that:
as the number of delay variables increases, the error increases, but the convergence rate hardly changes. 
By (\ref{mumuer}), the numerical experiment is consistent with the theoretical analysis.

\begin{figure}[htbp]
	\centering
	\begin{minipage}{0.49\linewidth}
		\centering
		\includegraphics[width=0.9\linewidth]{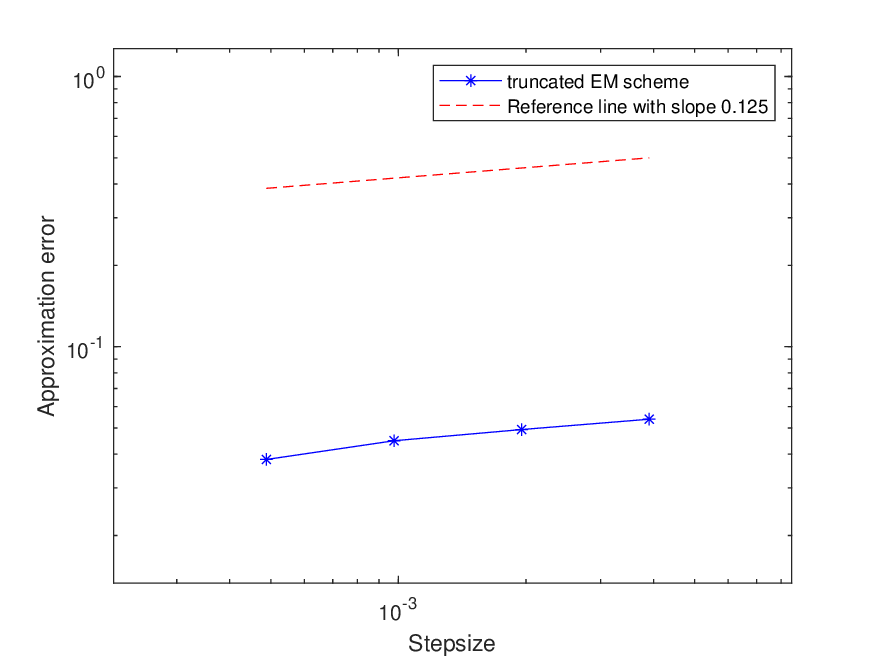}
		\caption{\label{tu1} Convergence rate of TEMS for (\ref{vq2})}
	\end{minipage}
	\begin{minipage}{0.49\linewidth}
		\centering
		\includegraphics[width=0.9\linewidth]{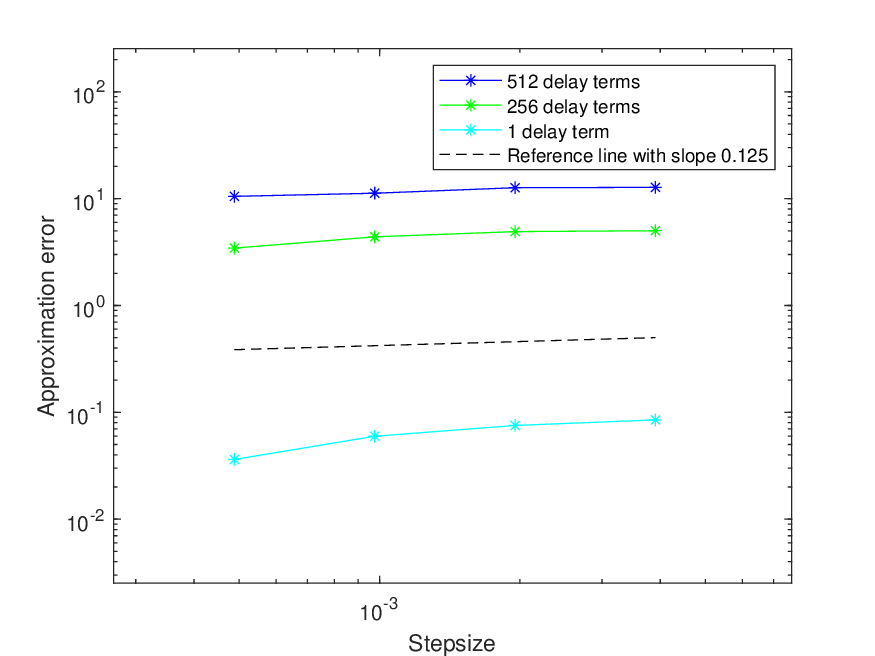}
		\caption{\label{tu2} Convergence rates of TEMS for (\ref{mul1})-(\ref{mul3})}
	\end{minipage}
\end{figure}


\end{document}